\documentclass[11pt]{article}
\usepackage{amssymb}
\usepackage[margin=1.4in]{geometry}

\usepackage{graphicx}
\usepackage{color}
\usepackage{amsmath}

\newtheorem{theorem}{Theorem}

\newtheorem{corollary}[theorem]{Corollary}

\newtheorem{definition}[theorem]{Definition}
\newtheorem{example}[theorem]{Example}

\newtheorem{proposition}[theorem]{Proposition}
\newtheorem{remark}[theorem]{Remark}

\newenvironment{proof}[1][Proof]{\textbf{#1.} }{\ \rule{0.5em}{0.5em}}

\def\essinf{\operatorname{ess.\!inf}}
\def\esssup{\operatorname{ess.\!sup}}
\def\essmax{\operatorname{ess.\!max}}

\begin{document}

\title{Fully-dynamic risk measures: horizon risk, time-consistency, and relations with BSDEs and BSVIEs}
\author{Giulia Di Nunno\thanks{Department of Mathematics,
University of Oslo, P.O. Box 1053 Blindern, N-0316 Oslo Norway.
Email: giulian@math.uio.no}
\thanks{NHH - Norwegian School of Economics, Helleveien 30, N-5045 Bergen, Norway.}
  \and Emanuela Rosazza Gianin\thanks{Department of Statistics and Quantitative Methods,
University of Milano Bicocca, via Bicocca degli Arcimboldi 8, 20126 Milano Italy.
Email:  emanuela.rosazza1@unimib.it}}
\date{November 15, 2023}

\maketitle
\begin{abstract}
In a dynamic framework, we identify a new concept associated with the risk of assessing the financial exposure by a measure that is not adequate to the actual time horizon of the position. This will be called {\it horizon risk}. We clarify that {\it dynamic risk measures} are subject to horizon risk, so we propose to use the {\it fully-dynamic} version. To quantify horizon risk, we introduce {\it h-longevity} as an indicator.
We investigate these notions together with other properties of risk measures as normalization, restriction property, and different formulations of time-consistency. We also consider these concepts for fully-dynamic risk measures generated by backward stochastic differential equations (BSDEs), backward stochastic Volterra integral equations (BSVIEs), and families of these. Within this study, we provide new results for BSVIEs such as a converse comparison theorem and the dual representation of the associated risk measures.

\vspace{2mm}\noindent

\vspace{2mm}\noindent
\end{abstract}

\section{Introduction}

Monetary risk assessment spans across time horizons with different length, from the very short ones for trading operations to the decades-long ones typical of sovereign wealth or pension funds.
Hence the use of adequate risk measures across time as well as a time-consistent evaluation of risk are important factors in risk quantification and management.

We address the problem of using an appropriate risk evaluation for the given time horizon. To explain, it is not correct to use a risk measure designed for long term positions to evaluate risks that occur in the short term.
This is considered a ``common error'' in risk quantification by some quants, see e.g.~\cite{DDB}.
Indeed, bearing in mind that the purpose of risk assessment and management is not avoiding risk, but giving the rightful possibility to invest in risky assets with a reasonable control on what is acceptable to the investor, then it is clear that this form of ``horizon risk'' should be tackled.
This is particularly evident in life insurance with an impact on pension funds and health insurance. For example, a wrong estimate of the individual remaining lifetime brings to a wrong use of the associated mortality tables and, consequently, of the risk of an underestimation of insurance premia and, by domino effect, a wrong estimation of capital requirements and their associated financial management.
In
this work, we identify {\it horizon risk} and we propose a way to quantify this via the {\it horizon longevity},
or {\it h-longevity} in short.
For this we work with the general framework of fully-dynamic risk measures that embeds the classical  dynamic risk measures. Indeed, {\it a fully-dynamic risk measure} is a family of risk measures $(\rho_{su})_{s,u} \triangleq (\rho_{su})_{0\leq s \leq u\leq T}$  $(\text{with } T<\infty)$ indexed by two time parameters, the first represents the evaluation time, the second the horizon to which the risk pertains. Instead, {\it a dynamic risk measure} is a family of risk measures $(\rho_{s})_{s} \triangleq (\rho_{s})_{0\leq s  \leq T}$ indexed by only one time parameter representing the time of risk evaluation.
So we have that
\begin{equation}
\label{eq:intro1}
\rho_s = \rho_{sT}.
\end{equation}
Whenever we consider a risk $X$ occurring at time $t\in [0,T]$ (hence $\mathcal{F}_t$-measurable for some given information flow), we can consider different risk evaluations at $s$, either $\rho_{st}(X)$ or $\rho_{su}(X)$, for any other horizon $u>t$ (even much later than $t$), since $\mathcal{F}_t \subseteq \mathcal{F}_u$.
Horizon longevity penalizes the use of the ``wrong'' risk measure, that is the one that does not pertain to the right time horizon.

The impact of the length of the time horizon on risk evaluation has already been detected in the work~\cite{FKN} on epistemic uncertainty. There, it comes out that ambiguity on the model choice is growing with the time horizon. While~\cite{FKN} is not related with our considerations, it shows anyhow that the time horizon effects the precision on risk quantification.

Another concept under investigation in this paper is {\it time-consistency}, which has different formulations and it plays an important role in dynamic risk evaluation. See e.g.~\cite{acciaio-penner,bielecki-etal,bion-nadal-SPA,bion-nadal-di-nunno}.

In our study we highlight that both h-longevity and time-consistency are linked to the {\it restriction property} (see~\cite{bion-nadal-di-nunno}), that is
\begin{equation}
\label{restriction}
\rho_{st}(X) = \rho_{su}(X), \qquad \text{for} \;\mathcal{F}_t\text{-measurable} \; X, \quad u \geq t.
\end{equation}
Indeed, the presence of restriction induces that the different formulations of time-consistency are equivalent, while the absence of restriction provides the very possibility to introduce h-longevity.
The restriction property is naturally satisfied by dynamic risk measures $(\rho_s)_s$ (see~\eqref{eq:intro1}) and this may be a reason for which the horizon risk has not been earlier identified in its own being and hence quantified. Similarly, also {\it normalization}, i.e.
\begin{equation}
\label{normalization}
\rho_{su}(0) = 0,
\end{equation}
is discovered to play a crucial role in the relationships among the concepts above. Note that normalization is often standardly assumed in many financial risk evaluations. However, it is questionable in the evaluation of environmental strategies, namely a non-action (zero investment) can carry a positive risk. Note that, by ESG regulations, these positions and their entailed risks have to be reported.

In the first part of this paper, we work in full generality with an axiomatic setting. In the second part we concentrate on risk measures induced by backward dynamics.
We focus on the Brownian framework (see, e.g.,~\cite{barrieu-el-karoui, jiang, peng05, rg}) and use a BSDE with driver $g$ to generate a fully-dynamic risk measure
$(\rho_{su})_{s,u}$. Then we study how the properties of the driver are connected to the concepts of time-consistency, h-longevity, restriction, and normalization.

Since a fully-dynamic risk measure $(\rho_{su})_{s,u}$ depends on the horizons $u \in [0,T]$, then we can also generate it from a {\it family} of BSDEs with drivers $\mathcal{G}= (g_u)_u$. In this way it is emphasized that each period $[s,u]$ is associated to the BSDE with driver $g_u$ providing the risk evaluation $\rho_{su}$.
In this framework we investigate again the relationships among the concepts of interest with the family of drivers. In this context we prove a Comparison Theorem for BSDEs with different time horizons useful for the study of h-longevity.

Recently, backward stochastic Volterra integral equations (BSVIEs) have been suggested to generate dynamic risk measures, see~\cite{yong07, agram}. However, at that stage, some important results on BSVIEs were still lacking. Here we provide a dual representation of risk measures induced by BSVIEs and a converse comparison theorem for BSVIEs. With these two results we have a full picture of the relationship between the properties of the driver and those of the corresponding risk measures as well as a comprehensive study of time-consistency and horizon longevity in the Volterra setting.

We recall that a BSVIE is induced by a family of BSDEs, see~\cite{yong13-survey}. This has been a further motivation to study fully-dynamic risk measures induced by a BSVIE with Volterra driver $g=g(t,s,\cdot)$, $s\geq t$. As a final stage we have also considered {\it families} of BSVIEs with Volterra drivers $\mathcal{G} = (g_u)_u$  again to emphasize the role of the time horizon in the risk evaluation.

Working with families of backward equations has made evident the surprisingly crucial role of the restriction property in risk measurement across time. For example, it turns out that the only way to have a strong (recursive) time-consistency is to work with fully-dynamic risk measures induced by a single standard BSDE.

Summarizing, in this work we have first identified horizon risk in its own being and proposed one way to quantify it. We have hence considered fully-dynamic risk measures and studied  the different forms of time-consistency and h-longevity, showing how the restriction property is actually playing a crucial role in these matters. Our work has covered both the general axiomatic approach and the risk measures generated by backward dynamics, both single and in families, both of standard and of Volterra type.
In this we have seen how the properties of the drivers reflect the properties of the fully-dynamic risk measures in respect of the risk evaluations across time. In between we have also provided results on BSDE and BSVIEs of independent interest.


\section{Fully-dynamic risk measures, time-consistency, and horizon risk}

In the sequel, we will focus on fully-dynamic risk measures that have been recently introduced by Bion-Nadal and Di Nunno~\cite{bion-nadal-di-nunno}. See also~\cite{BN-BMO}, where this concept was actually simply called dynamic risk measure.

\begin{definition}
A \emph{fully-dynamic risk measure} is a family $(\rho_{st})_{s,t}=(\rho_{st})_{0 \leq s \leq t \leq T}$ of risk measures indexed by two time parameters
$$\rho_{st}: L^{p}(\mathcal{F}_t) \to L^{p}(\mathcal{F}_s), \mbox{ with } p \in [1,+\infty],$$
that for all $X,Y \in L^{p}(\mathcal{F}_t)$ are:
\begin{itemize}
\item[-] monotone: $\rho_{st}(X) \leq \rho_{st}(Y)$ for $X \geq Y$,
\item[-] convex: $\rho_{st}(\lambda X + (1-\lambda) Y) \leq \lambda\rho_{st}(X)+(1-\lambda) \rho_{st}(Y)$ for all $\lambda \in [0,1]$,
\item[-] $\mathcal{F}_s$-translation invariant (cash-invariant): $\rho_{st}( X +  Y)=\rho_{st}(X)-Y$ for all $Y \in L^{p}(\mathcal{F}_s)$
\end{itemize}
and, for $p=\infty$,
\begin{itemize}
\item[-] continuous from below: if $X_n \uparrow X$ as $n \to + \infty$, then $\rho_{st}(X_n) \to \rho_{st} (X)$ in $L^{p}(\mathcal{F}_s)$, as $n \to + \infty$.
\end{itemize}
\end{definition}
In the definition above, continuity from below is assumed only in the case where $p=\infty$.
In fact, for any $p\in [1, \infty)$, it is implied by the other assumptions, see~\cite[Remark 2.5]{bion-nadal-di-nunno}. Furthermore, fully-dynamic risk measures satisfy {\it weak $\mathcal{F}_s$-homogeneity}, i.e.
$$
1_A \rho_{st}(X)= 1_A \rho_{st}(1_A X), \quad \mbox{ for any } X \in L^{p}(\mathcal{F}_t), A \in \mathcal{F}_s
$$
(see~\cite[Remark 2.6]{bion-nadal-di-nunno}) and have the following dual representations
\begin{eqnarray}
\rho_{st}(X)&=& \underset{Q \in \mathcal{Q}_{st}}{\essmax} \;
 \{E_Q [-X | \mathcal{F}_s] - \alpha_{st}(Q) \}
 \label{eq: dual-repres-fully-dyn} \\
&=& \underset{\substack{Q \ll P : \\  E_P [\alpha_{st}(Q)] < \infty }}{\essmax}
\{E_Q [-X | \mathcal{F}_s] - \alpha_{st}(Q) \}
 \label{eq: dual-repres-fully-dyn-2}
\end{eqnarray}
where
\begin{equation}
\label{eq: Q_tu - gt}
\mathcal{Q}_{st} \triangleq \left\{Q \mbox{ on } \mathcal{F}_t: Q \ll P \mbox{ and } Q|_{\mathcal{F}_s} \equiv P|_{\mathcal{F}_s} \right\}
\end{equation}
and
\begin{equation}
\label{minimal penalty}
\alpha_{st}(Q)\triangleq \underset{X\in L^p(\mathcal{F}_t)}{\esssup} \;
 \{E_Q [-X | \mathcal{F}_s] - \rho_{st}(X) \}
\end{equation}
is the {\it minimal penalty functional}. See~\cite[Proposition 2.8]{bion-nadal-di-nunno} for details.

We stress that, in general, the risk measures in the family of the fully dynamic risk measure are not normalized, i.e. $\rho_{st}(0) \neq 0$, and the family does not satisfy the restriction property, that is $\rho_{rt}(Y) \neq \rho_{rs}(Y)$ for $Y \in L^{p}(\mathcal{F}_s)$, for $s \leq t$.
\bigskip

In view of the dynamic nature of the risk evaluation, we consider the relationships among the inter-temporal evaluations given by the fully-dynamic risk measure.
For this, we recall here below the notions of time-consistency mostly used in the literature and their connections expressed in the present setting. See, e.g.,~\cite{acciaio-penner, bielecki-etal, bion-nadal-SPA}, among others.

\begin{definition}
\hspace{10cm}
\begin{itemize}
  \item \emph{Strong time-consistency} (or \emph{recursivity}): for any $s, t, u \in [0,T]$ with $s \leq t \leq u$,
  \begin{equation}
  \label{strong tc}
  \rho_{st} (-\rho_{tu}(X))= \rho_{su}(X) \quad \mbox{ for any } X \in L^{p}(\mathcal{F}_u).
  \end{equation}
  \item \emph{Order time-consistency}: for any $s, t, u \in [0,T]$ with $s \leq t \leq u$,
  \begin{equation}
    \label{order tc}
  \rho_{tu}(X)=\rho_{tu}(Y), \quad X,Y \in L^p(\mathcal{F}_u) \Longrightarrow \rho_{su}(X)=\rho_{su}(Y) .
  \end{equation}
  \item \emph{Weak time-consistency}: for any $s, t, u \in [0,T]$ with $s \leq t \leq u$,
  \begin{equation}
  \label{eq: weak-tc}
  \rho_{su} (\rho_{tu}(0)-\rho_{tu}(X))= \rho_{su}(X) \quad \mbox{ for any } X \in L^{p}(\mathcal{F}_u).
  \end{equation}
  \end{itemize}
\end{definition}

\noindent Observe that strong time-consistency is equivalent to the {\it cocycle condition} on the minimal penalties:
\begin{equation} \label{eq: cocycle}
\alpha_{su} (Q) = \alpha_{st}(Q|_{\mathcal{F}_t})+ E_{Q} [\alpha_{tu}(Q)|\mathcal{F}_s ], \quad \forall s,t,u, Q \in \mathcal{Q}_{su},
\end{equation}
together with the {\it m-stability} of the Radon-Nykodym derivatives associated to the corresponding sets of measures $(\mathcal{Q}_{st})_{s,t}$, i.e. the probability measure $S$ defined by
\begin{equation}
\label{pasting}
\frac{dS}{dP} = \frac{dQ}{dP} \, \frac{dR}{dP}, \quad \textrm{for all } Q\in \mathcal{Q}_{st}, R \in \mathcal{Q}_{tu}
\end{equation}
 belongs to $\mathcal{Q}_{su}$.
See~\cite[Theorem 2.5]{BN-BMO}.
The operation in~\eqref{pasting} is called {\it pasting} of probability measures, also shortly denoted $Q \cdot R$.
Naturally, any measure $S \in \mathcal{Q}_{su}$ admits a representation in the form~\eqref{pasting}. In fact, we have
$$
\frac{d S}{dP} = \frac{dS|_{\mathcal{F}_t}}{dP} \, \frac{dR}{dP} \quad \textrm{with }  \frac{dR}{dP} = \frac{ \frac{dS}{dP} }{E_P \big[\frac{dS}{dP}\vert \mathcal{F}_t \big] } 1_A + 1_{A^c}
 $$
where $A\triangleq \big\{ \omega\in \Omega : E_P \big[\frac{dS}{dP}\vert \mathcal{F}_t\big](\omega) > 0 \big\}$.

We recall from Proposition 2.13 and Corollary 2.14 of Bion-Nadal and Di Nunno~\cite{bion-nadal-di-nunno} that, for fully-dynamic risk measures, strong time-consistency implies order time-consistency; furthermore, for {\it normalized} fully-dynamic risk measures strong time-consistency is equivalent to the restriction property plus order time-consistency.

It is worth emphasizing that, differently from strong time-consistency, the formulation of weak time-consistency is in terms of risk measures with the same time horizon $u$.
The notion of weak time-consistency is motivated by the discussion done by Bion-Nadal and Di Nunno~\cite{bion-nadal-di-nunno}, where it was shown that risk indifference pricing does not satisfy strong time-consistency, but only a weaker version.

Notice that under the additional assumption of normalization~\eqref{normalization}, weak time-consistency~\eqref{eq: weak-tc} becomes
\begin{equation} \label{eq: weak-tc-normalized}
  \rho_{su} (-\rho_{tu}(X))= \rho_{su}(X) \quad \mbox{ for any } X \in L^{p}(\mathcal{F}_u),
  \end{equation}
while, under both normalization and the restriction property, weak time-consistency~\eqref{eq: weak-tc} reduces to the classical strong time-consistency~\eqref{strong tc}.

\bigskip
Again from~\cite{bion-nadal-di-nunno} and from the arguments above, it emerges that the restriction property has not only an extremely important role on time-consistency but also on risk measures. Indeed, once the restriction property is dropped, we can explicitly quantify
the increasing cost of riskiness for longer time horizons.

\begin{definition}\label{def: horizon longevity}
A fully-dynamic risk measures $(\rho_{st})_{s,t}=(\rho_{st})_{0 \leq s \leq t \leq T}$ is said to take into account \emph{Horizon longevity} (or \emph{h-longevity}, for short) if for any fixed $s \in [0,T]$,
  \begin{equation} \label{eq: time-value}
  \rho_{st} (X) \leq \rho_{su}(X) \quad \mbox{ for any }  0 \leq t \leq u, \; X \in L^{p}(\mathcal{F}_t) .
  \end{equation}
  Naturally, a risk measure satisfying restriction property (see~\eqref{restriction}) trivially satisfies h-longevity. Then we talk about \emph{strict h-longevity} to exclude this trivial case.
\end{definition}

Differently from weak time-consistency~\eqref{eq: weak-tc} that has an impact on the first time parameter, horizon longevity~\eqref{eq: time-value} focuses on the behavior of the second time parameter, that is the time horizon. Quantifying horizon risk is particularly important in life insurance and pension funds, as reported in the introduction of this work. 
\medskip

Our aim is to investigate weak time-consistency and h-longevity on fully-dynamic risk measures $(\rho_{st})_{0 \leq s \leq t \leq T}$ from an axiomatic point of view as well as their impact on the relation between fully-dynamic risk measures and BSDEs or BSVIEs.

\subsection{Time-consistency}

The next results focus on weak time-consistency and on its characterizations.
Also we introduce the new concept of {\it sub time-consistency}, see Proposition~\ref{prop: charact-penalty-weak-tc}, which will turn out to interplay with h-longevity.

\begin{proposition}
\label{prop: equiv-weak-tc}
A fully-dynamic risk measure $(\rho_{st})_{0 \leq s \leq t \leq T}$ is weakly time-consistent if and only if it satisfies order time-consistency.
\end{proposition}

\noindent
\begin{proof}
Assume order time-consistency. Let $X \in L^p(\mathcal{F}_u)$ and $t \leq u$ be arbitrary. By $\mathcal{F}_t$-translation invariance of $\rho_{tu}$ it follows
\begin{equation} \label{eq: eq000}
\rho_{tu}(X)=\rho_{tu}(0-\rho_{tu}(X))-\rho_{tu}(0)= \rho_{tu}(\rho_{tu}(0)-\rho_{tu}(X)).
\end{equation}
Since $Y \triangleq \rho_{tu}(0)-\rho_{tu}(X) \in L^p(\mathcal{F}_t)$,~\eqref{eq: eq000} and order time-consistency imply that $\rho_{su}(X)=\rho_{su}(Y)$ for any $s \leq t$. The thesis is therefore proved.

Conversely, suppose that $\rho_{tu}(X)=\rho_{tu}(Y)$ holds for some $X,Y \in L^p(\mathcal{F}_u)$ and $t \leq u$. By weak time-consistency it follows that
$$
\rho_{su}(X)= \rho_{su} (\rho_{tu}(0)-\rho_{tu}(X))=\rho_{su} (\rho_{tu}(0)-\rho_{tu}(Y))=\rho_{su}(Y)
$$
for any $s \leq t$.
\end{proof}

\noindent
Naturally, since strong time-consistency implies order time-consistency, the result above gives that strong time-consistency implies also weak time-consistency.

\begin{proposition}
\label{prop: charact-penalty-weak-tc}
Let $(\alpha_{st})_{0 \leq s \leq t \leq T}$ be the minimal penalty terms of $(\rho_{st})_{0 \leq s \leq t \leq T}$.\smallskip

\noindent a)
If $(\alpha_{ru})_{0 \leq r \leq u \leq T}$ satisfies, for all $s\leq t\leq u$,
\begin{equation}
\label{eq: characteriz-weak-tc-penalty}
\alpha_{su} (S)= \alpha_{su}(Q)+ E_{Q} [\alpha_{tu}(R) - \essinf_{\bar{R} \in \mathcal{Q}_{tu}} \alpha_{tu} (\bar{R}) |\mathcal{F}_s ],
\end{equation}
 for all $Q \in \mathcal{Q}_{su}, R \in {Q}_{tu},$
where $S=Q|_{\mathcal{F}_t} \cdot R \in \mathcal{Q}_{su}$ is obtained by pasting $Q$ on $[s,t]$ and $R$ on $[t,u]$, then $(\rho_{st})_{0 \leq s \leq t \leq T}$ is weakly time-consistent.\smallskip

\noindent b)
Weak time-consistency of $(\rho_{st})_{0 \leq s \leq t \leq T}$ implies that, for all $s\leq t\leq u$,
\begin{equation} \label{eq: characteriz-weak-tc-penalty-2}
\alpha_{su} (S) \leq \alpha_{su}(Q)+ E_{Q} [\alpha_{tu}(R) - \essinf_{\bar{R} \in \mathcal{Q}_{tu}} \alpha_{tu} (\bar{R}) |\mathcal{F}_s ],
\end{equation}
for all $Q \in \mathcal{Q}_{su}, R \in {Q}_{tu},$
where $S=Q|_{\mathcal{F}_t} \cdot R \in \mathcal{Q}_{su}$ is obtained by pasting $Q$ on $[s,t]$ and $R$ on $[t,u]$. In~\eqref{eq: characteriz-weak-tc-penalty-2} equality holds at least for the optimal scenarios in the dual representation~\eqref{eq: dual-repres-fully-dyn}-\eqref{eq: dual-repres-fully-dyn-2}.
\end{proposition}

It is worth emphasizing that~\eqref{eq: characteriz-weak-tc-penalty} is different from the usual cocycle condition~\eqref{eq: cocycle} for two reasons. First, it depends on the additional term $\essinf_{\bar{R} \in \mathcal{Q}_{tu}} \alpha_{tu} (\bar{R})$, which is due to the non-normalization of the risk measure. Second, both $\alpha_{su} (S)$ and $\alpha_{su} (Q)$ refer to the time horizon $u$.
Note that $\essinf_{\bar{R} \in \mathcal{Q}_{tu}} \alpha_{tu} (\bar{R}) \in L^p(\mathcal{F}_t)$ because of $\rho_{tu}(0)=-\essinf_{\bar{R} \in \mathcal{Q}_{tu}} \alpha_{tu} (\bar{R})$ and, by assumption, $\rho_{tu}(0) \in L^p(\mathcal{F}_t)$.

\smallskip
In terms of the {\it normalized risk measure}
$$
\widetilde{\rho}_{st} (X) \triangleq \rho_{st} (X) - \rho_{st} (0)
$$
and of its minimal penalty function
$$
\widetilde{\alpha}_{st} (Q) \triangleq \alpha_{st} (Q) - \essinf _{\bar{R} \in \mathcal{Q}_{st}} \alpha_{st} (\bar{R}),
$$
condition~\eqref{eq: characteriz-weak-tc-penalty} becomes
\begin{equation} \label{eq: characteriz-weak-tc-penalty-tilde}
\widetilde{\alpha}_{su} (S)= \widetilde{\alpha}_{su}(Q)+ E_{Q} [\widetilde{\alpha}_{tu}(R) |\mathcal{F}_s ], \quad \forall s,t,u, Q \in \mathcal{Q}_{su}, R \in {Q}_{tu}.
\end{equation}
Notice also that while weak time-consistency of a fully-dynamic risk measure $(\rho_{st})_{0 \leq s \leq t \leq T}$ is equivalent to that of the corresponding normalized $(\widetilde{\rho}_{st})_{0 \leq s \leq t \leq T}$,
the same is no more true for h-longevity. The reason is that the normalization terms $\rho_{st}(0)$ and $\rho_{su}(0)$ are different in general. As pointed out by Bion-Nadal and Di Nunno~\cite{bion-nadal-di-nunno}, Remark 2.12, also strong time-consistency of a fully-dynamic risk measure is not transferred to the normalized version.

\smallskip
\noindent
\begin{proof}[Proof of Proposition~\ref{prop: charact-penalty-weak-tc}]
In view of the above comments, for simplicity, we prove the result in terms of the normalized $\widetilde{\rho}$.

\noindent
a)
Assume that~\eqref{eq: characteriz-weak-tc-penalty-tilde} holds true. By the dual representation~\eqref{eq: dual-repres-fully-dyn} of $\widetilde{\rho}$,
\begin{eqnarray}
&&\widetilde{\rho}_{su} (-\widetilde{\rho}_{tu}(X)) \notag \\
&=&
 \underset{Q \in \mathcal{Q}_{su}}{\essmax} \;\{E_Q [\widetilde{\rho}_{tu}(X)  | \mathcal{F}_s] - \widetilde{\alpha}_{su}(Q) \} \notag\\
&=&
 \underset{Q \in \mathcal{Q}_{su}}{\essmax} \;
 \{E_Q [
  \underset{R \in \mathcal{Q}_{tu}}{\essmax} \;
 \{E_R [-X  | \mathcal{F}_t] - \widetilde{\alpha}_{tu}(R) \}  | \mathcal{F}_s] - \widetilde{\alpha}_{su}(Q) \} \notag\\
&=&
 \underset{Q \in \mathcal{Q}_{su}, R \in \mathcal{Q}_{tu}}{\essmax} \;
 \{E_Q [E_R [-X  | \mathcal{F}_t] | \mathcal{F}_s] - E_Q [\widetilde{\alpha}_{tu}(R) | \mathcal{F}_s]- \widetilde{\alpha}_{su}(Q) \}  \label{eq; eq003a}\\
&=&
 \underset{\substack{S=Q|_{\mathcal{F}_t} \cdot R: \\ Q \in \mathcal{Q}_{su}, R \in \mathcal{Q}_{tu}}}{\essmax} \;
 \{E_S [-X   | \mathcal{F}_s] -  \widetilde{\alpha}_{su}(S) \} = \widetilde{\rho}_{su}(X), \label{eq; eq003b}
\end{eqnarray}
where $S$ is obtained by pasting $Q$ on $[s,t]$ and $R$ on $[t,u]$.
Here above,~\eqref{eq; eq003a} follows from the same arguments as in~\cite{bion-nadal-SPA} and~\cite{detlef-scandolo}, while the first equality in~\eqref{eq; eq003b}  is due to~\eqref{eq: characteriz-weak-tc-penalty-tilde} and to m-stability~\eqref{pasting}.
\smallskip

\noindent
b)
Assume that weak time-consistency~\eqref{eq: weak-tc-normalized} holds for $(\widetilde{\rho}_{st})_{0 \leq s \leq t \leq T}$. On the one hand, again by the dual representation~\eqref{eq: dual-repres-fully-dyn}, we have
\begin{eqnarray}
\hspace{-4mm} &&\widetilde{\rho}_{su} (-\widetilde{\rho}_{tu}(X)) \notag \\
\hspace{-4mm}&=&
\underset{Q \in \mathcal{Q}_{su} ,R \in \mathcal{Q}_{tu}}{\essmax} \;
 \{E_Q [E_R [-X  | \mathcal{F}_t] | \mathcal{F}_s] - E_Q [\widetilde{\alpha}_{tu}(R) | \mathcal{F}_s]- \widetilde{\alpha}_{su}(Q) \}.   \label{eq: proof1}
\end{eqnarray}
On the other hand, by weak time-consistency,
\begin{eqnarray}
&&\widetilde{\rho}_{su} (-\widetilde{\rho}_{tu}(X))=\widetilde{\rho}_{su} (X)
= \underset{S \in \mathcal{Q}_{su}}{\essmax}\; \{E_S [-X   | \mathcal{F}_s] - \widetilde{\alpha}_{su}(S) \} \notag \\
&=& \underset{\substack{
S=Q|_{\mathcal{F}_t} \cdot R: \\ Q \in \mathcal{Q}_{su} ,R \in \mathcal{Q}_{tu}
}}{\essmax}\; \{E_Q [E_R [-X  | \mathcal{F}_t] | \mathcal{F}_s] - \widetilde{\alpha}_{su}(S) \},\label{eq: proof2}
\end{eqnarray}
which is due to the fact that, for any $S \in \mathcal{Q}_{su}$, there exist $Q \in \mathcal{Q}_{su}, R \in \mathcal{Q}_{tu}$ such that $S=Q|_{\mathcal{F}_t} \cdot R$ and, vice versa, given $Q \in \mathcal{Q}_{su},R \in \mathcal{Q}_{tu}$, the pasting $S=Q|_{\mathcal{F}_t} \cdot R \in \mathcal{Q}_{su}$.
Since $\widetilde{\alpha}_{su}(S)$ is the minimal penalty function on $[s,u]$,~\eqref{eq: proof1} and~\eqref{eq: proof2} imply that
\begin{equation}\label{eq:service}
\widetilde{\alpha}_{su} (S) \leq \widetilde{\alpha}_{su}(Q)+ E_{Q} [\widetilde{\alpha}_{tu}(R)  |\mathcal{F}_s ]
\end{equation}
for any $Q \in \mathcal{Q}_{su},R \in \mathcal{Q}_{tu}$ and the pasting $S=Q|_{\mathcal{F}_t} \cdot R$.
\smallskip

It remains to prove that~\eqref{eq: characteriz-weak-tc-penalty-2} holds with an equality at least for the optimal scenarios in the dual representation. By the arguments above, it is then enough to prove the reverse inequality in~\eqref{eq:service}.
Since in the dual representation~\eqref{eq: dual-repres-fully-dyn} the $\essmax$ is attained, it follows that $\widetilde{\rho}_{su} (-\widetilde{\rho}_{tu}(X))= \widetilde{\rho}_{su}(X)$ becomes
\begin{equation} \label{eq: eq001}
E_{Q} [E_{R} [-X |\mathcal{F}_t ] - \widetilde{\alpha}_{tu}(R) |\mathcal{F}_s ] - \widetilde{\alpha}_{su}(Q)=\widetilde{\rho}_{su}(X)
\end{equation}
for some $R \in \mathcal{Q}_{tu}, Q \in \mathcal{Q}_{su}$ (depending on $X$). By defining $S$ as the pasting $Q \cdot R$ of $Q$ on $[s,t]$ and $R$ on $[t,u]$,~\eqref{eq: eq001} reduces to
\begin{equation} \label{eq: eq002}
E_{S} [-X |\mathcal{F}_s ]- E_{Q} [ \widetilde{\alpha}_{tu}(R) |\mathcal{F}_s ] - \widetilde{\alpha}_{su}(Q)=\widetilde{\rho}_{su}(X).
\end{equation}
From $\widetilde{\rho}_{su}(X) \geq E_{S} [-X |\mathcal{F}_s ] - \widetilde{\alpha}_{su}(S)$ and~\eqref{eq: eq002}, it follows that $\widetilde{\alpha}_{su} (S) \geq \widetilde{\alpha}_{su}(Q)+ E_{Q} [\widetilde{\alpha}_{tu}(R) |\mathcal{F}_s ]$.
\end{proof}

\begin{remark} \label{rem: tc and longevity}
If $(\rho_{st})_{0 \leq s \leq t \leq T}$ satisfies weak time-consistency, h-longevity, and $\rho_{tu}(0) \leq 0$ for any $t,u$, then
\begin{equation} \label{eq: sub TC}
\rho_{st} (-\rho_{tu}(X)) \leq \rho_{su}(X) \quad \mbox{ for any } X \in L^{p}(\mathcal{F}_u),
\end{equation}
called sub (strong) time-consistency in the following. Notice that $\rho_{tu}(0) \leq 0$ is equivalent to $\essinf_{Q \in \mathcal{Q}_{tu}} \alpha_{tu} (Q)\geq 0$.

On the other hand, sub (strong) time-consistency~\eqref{eq: sub TC} together with $\rho_{tu}(0) \geq 0$, for any $t,u$, imply h-longevity. For any $\bar{X} \in L^{p}(\mathcal{F}_t)$, indeed,
$$
\rho_{st} (\bar{X}) \leq \rho_{st} (\bar{X}-\rho_{tu}(0)) = \rho_{st} (-\rho_{tu}(\bar{X}))  \leq \rho_{su}( \bar{X}),
$$
where the first inequality is due to $\rho_{tu}(0) \geq 0$ and monotonicity, the last to sub time-consistency, while the equality follows by $\mathcal{F}_t$-translation invariance.
\end{remark}

Proceeding similarly to Acciaio and Penner~\cite{acciaio-penner}, it is easy to prove the relation between sub time-consistency and acceptance sets for fully-dynamic risk measures, here formulated as follows.

\begin{proposition} \label{prop: charact-sub time-cons}
For any fully-dynamic risk measure $(\rho_{st})_{0 \leq s \leq t \leq T}$ the following are equivalent:\smallskip

\noindent (i)
Sub (strong) time-consistency, that is $\rho_{st} (-\rho_{tu}(X)) \leq \rho_{su}(X)$ for any $X \in L^{p}(\mathcal{F}_u)$, $s \leq t \leq u$;
\smallskip

\noindent (ii)
$\mathcal{A}_{su} \subseteq \mathcal{A}_{st} + \mathcal{A}_{tu}$ where $\mathcal{A}_{su} \triangleq \{Z \in L^{p}(\mathcal{F}_u): \rho_{su} (Z) \leq 0, P\mbox{-a.s.} \}$ is the acceptance set associated to $\rho_{su}$;
\smallskip

\noindent (iii)
$\alpha_{su}(Q) \leq \alpha_{st} (Q|_{\mathcal{F}_t}) + E_Q[\alpha_{tu}(Q) \vert \mathcal{F}_s]$ for any $Q \in \mathcal{Q}_{su}$, $s \leq t \leq u$.
\end{proposition}

To summarize, the following implications among the different notions of time-consistency investigated above and h-longevity hold:\medskip

$$
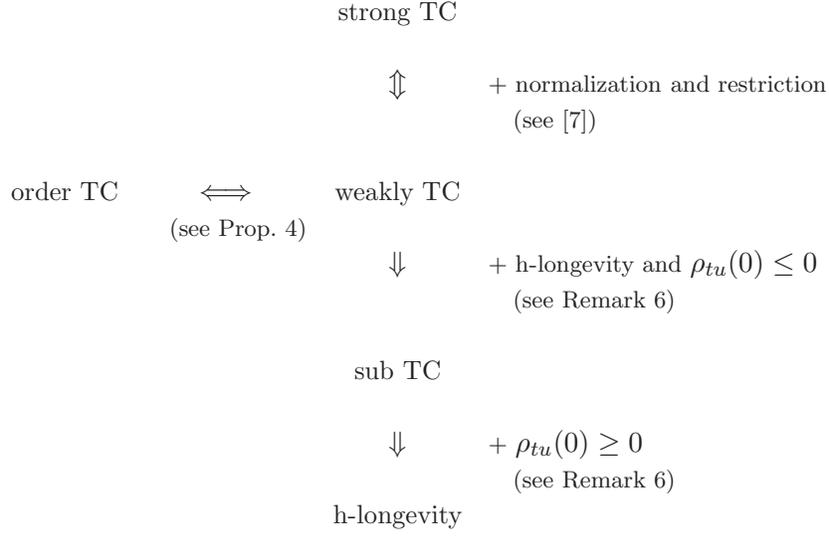
\begin{array}{cccl}
&&\mbox{\small strong TC} & \\
\\
&& \Updownarrow & {\footnotesize \mbox{+ normalization and restriction}}\\
&& & {\footnotesize \mbox{\quad (see~\cite{bion-nadal-di-nunno})}}\\\\
 \mbox{\small order TC} & \Longleftrightarrow & \mbox{\small weakly TC} & \\
& {\footnotesize\mbox{\quad (see Prop.~\ref{prop: equiv-weak-tc})}} && \\
&& \Downarrow  & {\footnotesize \mbox{+ h-longevity and }  \rho_{tu}(0) \leq 0  \mbox{ }}\\
&&&{\footnotesize \mbox{\quad (see Remark~\ref{rem: tc and longevity})}}\\
&&&\\
&& \mbox{\small sub TC} & \\
 \\
&& \Downarrow  & {\footnotesize \mbox{+  }  \rho_{tu}(0) \geq 0 } \\
&&   & \mbox{\footnotesize \quad (see Remark~\ref{rem: tc and longevity})}\\
&& \mbox{\small h-longevity} &
\end{array}
$$

\subsection{Horizon risk and h-longevity} \label{sec: longevity}

We recall that h-longevity~\eqref{eq: time-value} has been formulated as
\begin{equation*}
  \rho_{st} (X) \leq \rho_{su}(X) \quad \mbox{ for any }  0 \leq s \leq t \leq u,  X \in L^{p}(\mathcal{F}_t).
  \end{equation*}

From the formulation of h-longevity it is clear that the risk measure with longer horizon $\rho_{su}$ is relevant only restricted on $L^{p}(\mathcal{F}_t)$. In other words, longevity aims to penalize the risk measurement of a position done with the wrong risk measure, i.e. with the risk measure suitable for a longer horizon.

For this reason, in the sequel, we denote the restriction of $\rho_{su}$ on $L^{p}(\mathcal{F}_t)$ by $\bar{\rho}^t_{su}$, while $\bar{\alpha}^t_{su}$ is the corresponding minimal penalty function, i.e.
\begin{equation*}
\bar{\alpha}^t_{su}(Q)=\underset{X \in L^{p}(\mathcal{F}_t)}{\esssup}\; \{E_Q[-X | \mathcal{F}_s] - \bar{\rho}^t_{su}(X)\}, \quad Q \in \mathcal{Q}_{st}.
\end{equation*}
It then follows that
\begin{equation*}
\bar{\alpha}^t_{su}(Q_{\vert\mathcal{F}_t}) \leq \alpha_{su}(Q), \qquad Q\in \mathcal{Q}_{su}.
\end{equation*}
In fact,
\begin{eqnarray*}
\alpha_{su}(Q)&=&\underset{X \in L^{p}(\mathcal{F}_u)}{\esssup} \{E_Q[- |X \mathcal{F}_s] - \rho_{su}(X)\} \\
&\geq &\underset{X \in L^{p}(\mathcal{F}_t)}{\esssup} \{E_Q[-X | \mathcal{F}_s] -\bar{\rho}^t_{su}(X)\}=\bar{\alpha}^t_{su}(\left. Q \right|_{\mathcal{F}_t}).
\end{eqnarray*}

\bigskip
Using the notation above, h-longevity naturally leads to
\begin{equation} \label{eq: longevity-gammaX}
\bar{\rho}^t_{su}(X)=\rho_{su}(X)=\rho_{st}(X)+\gamma(s,t,u,X)
\end{equation}
for any  $0 \leq s \leq t \leq u$, $X \in L^{p}(\mathcal{F}_t)$, and
for a suitable $\mathcal{F}_s$-measurable $\gamma(s,t,u,X) \geq 0$ that may depend on the position $X$, on the time of evaluation $s$, on the time parameter $t$ referring to the measurability of $X$, and on the ``wrong'' time horizon $u$ used for evaluating $X$.
The term $\gamma$ can be then seen as an indicator quantifying the horizon risk or, roughly speaking, as an additive term of adjustment/calibration.
It follows that, for any $s \leq t \leq u$ and $X \in L^{p}(\mathcal{F}_t)$, $\gamma(s,t,t,X)=0$ and, by  translation invariance of the fully-dynamic risk measure, that $\gamma(s, t, u,X+c_s)=\gamma(s,t,u,X)$ for any $c_s \in L^p(\mathcal{F}_s)$.
The restriction property~\eqref{restriction} is then equivalent to $\gamma(s,t,u,X) = 0$ for all $0 \leq s \leq t \leq u$, $X \in L^{p}(\mathcal{F}_t)$.

\medskip

The following result characterizes longevity in terms of acceptance sets and of penalty functions.

\begin{proposition} \label{prop: longevity-character}
For a fully dynamic risk measure $(\rho_{st})_{s,t}$:

\noindent (a) longevity is equivalent to $\mathcal{A}_{su} \cap L^{p}(\mathcal{F}_t) \subseteq \mathcal{A}_{st}$ for any $s\leq t \leq u$.

\noindent (b) restriction is equivalent to $\mathcal{A}_{su} \cap L^{p}(\mathcal{F}_t) = \mathcal{A}_{st}$ for any $s\leq t \leq u$.

\noindent (c) longevity implies that $\alpha_{st}(Q) \geq \bar{\alpha}^t_{su}(Q)$ and also
$$\bar{\alpha}^t_{su}(Q ) \geq \alpha_{st}(Q )-\esssup_{X \in L^{p}(\mathcal{F}_t)} \gamma(s,t,u,X),$$
 for any $s\leq t \leq u$ and $Q\in \mathcal{Q}_{st}$.

\noindent (d) restriction implies that $\alpha_{st}(Q) = \bar{\alpha}^t_{su}(Q)$ for any $s\leq t \leq u$ and $Q \in \mathcal{Q}_{st}$.
\end{proposition}

\begin{proof}
(a) For any $X \in \mathcal{A}_{su} \cap L^{p}(\mathcal{F}_t)$ it follows that $\rho_{su}(X)\leq 0$ and, by longevity, also that $\rho_{st}(X)\leq 0$. Hence $\mathcal{A}_{su} \cap L^{p}(\mathcal{F}_t) \subseteq \mathcal{A}_{st}$.
Conversely, from the representation of risk measures in terms of acceptance sets and from the condition on acceptance sets, it follows that, for any $X \in L^{p}(\mathcal{F}_t)$,
\begin{eqnarray*}
\rho_{su}(X)&=& \essinf \{m_s \in L^{p}(\mathcal{F}_s)| \; m_s + X \in \mathcal{A}_{su} \} \\
&\geq & \essinf \{m_s \in L^{p}(\mathcal{F}_s)| \; m_s + X \in \mathcal{A}_{st} \}=\rho_{st}(X).
\end{eqnarray*}

\noindent
(b) Assume now restriction. From~\eqref{eq: longevity-gammaX}, we see that restriction can be interpreted as h-longevity with $\gamma(s,t,u,X) =0$.  Hence, from (a), it remains to prove the reverse inclusion for acceptance sets.
For any $X \in \mathcal{A}_{st}$ (hence $X \in L^{p}(\mathcal{F}_t)$) we have that $\rho_{st} (X) \leq 0$. By restriction, $\rho_{su}(X)=\rho_{st}(X)\leq 0$, hence $\mathcal{A}_{su} \cap L^{p}(\mathcal{F}_t) \supseteq \mathcal{A}_{st}$.
Conversely, assume $\mathcal{A}_{su} \cap L^{p}(\mathcal{F}_t) = \mathcal{A}_{st}$. From the representation of risk measures in terms of acceptance sets, we obtain that, for any $X \in L^{p}(\mathcal{F}_t)$,
\begin{eqnarray*}
\rho_{su}(X)&=& \essinf \{m_s \in L^{p}(\mathcal{F}_s)| \: m_s + X \in \mathcal{A}_{su} \} \\
&= & \essinf \{m_s \in L^{p}(\mathcal{F}_s)| \:m_s + X \in \mathcal{A}_{st} \}=\rho_{st}(X).
\end{eqnarray*}

\noindent
(c) Proceeding similarly to F\"{o}llmer and Schied~\cite{follmer-schied-book}, Thm. 4.16, by  translation invariance of $\rho_{su}$, we have
\begin{equation*}
\alpha_{st}(Q)=\underset{X \in \mathcal{A}_{st}}{\esssup} \, E_Q [-X | \mathcal{F}_s]
\end{equation*}
for any $s\leq t$ and $Q\in \mathcal{Q}_{st}$.
By (a), h-longevity and translation invariance imply that, for any $Q\in \mathcal{Q}_{st}$,
\begin{eqnarray*}
\alpha_{st}(Q)&=& \underset{X \in \mathcal{A}_{st}}{\esssup} \,E_Q [-X | \mathcal{F}_s] \notag\\
&\geq & \underset{X \in \mathcal{A}_{su}\cap L^{p}(\mathcal{F}_t)}{\esssup} E_Q [-X | \mathcal{F}_s]\\
&=&\underset{\substack{X \in L^{p}(\mathcal{F}_t):\\ \rho_{su}(X) \leq 0}}{\esssup} E_Q [-X | \mathcal{F}_s] = \bar{\alpha}^t_{su}(Q).
\end{eqnarray*}
Moreover, in terms of the indicator $\gamma$ in~\eqref{eq: longevity-gammaX}, we have that
\begin{eqnarray*}
\bar{\alpha}^t_{su}( Q )
&= &\underset{X \in L^{p}(\mathcal{F}_t)}{\esssup} \{E_Q[-X | \mathcal{F}_s] -\bar{\rho}^t_{su}(X)\} \\
&=&\underset{X \in L^{p}(\mathcal{F}_t)}{\esssup}  \{E_Q[-X | \mathcal{F}_s] - \rho_{st}(X)-\gamma(s,t,u,X)\} \\
&\geq &\underset{X \in L^{p}(\mathcal{F}_t)}{\esssup}  \{E_Q[-X | \mathcal{F}_s] - \rho_{st}(X)\}-\underset{X \in L^{p}(\mathcal{F}_t)}{\esssup}  \gamma(s,t,u,X) \\
&=& \alpha_{st}( Q ) -\underset{X \in L^{p}(\mathcal{F}_t)}{\esssup}  \gamma(s,t,u,X).
\end{eqnarray*}

\noindent
(d) By the equivalence between restriction and the condition on acceptance sets in (b), it follows that for any $Q\in \mathcal{Q}_{st}$
\begin{eqnarray*}
\alpha_{st}(Q)&=& \underset{X \in \mathcal{A}_{st}}{\esssup} \,E_Q [-X | \mathcal{F}_s] \\
&= & \underset{X \in \mathcal{A}_{su}\cap L^{p}(\mathcal{F}_t)}{\esssup} \, E_Q [-X | \mathcal{F}_s]=\bar{\alpha}^t_{su}(Q).
\end{eqnarray*}
The proof is then complete. \end{proof}

\smallskip\noindent
With (c), we see that the greater is $\esssup_{X \in L^{p}(\mathcal{F}_t)} \gamma(s,t,u,X)$ for $\mathcal{F}_t$-measurable positions, the lower is the lower bound $\alpha_{st}(\left. Q \right|_{\mathcal{F}_t}) -\esssup_{X \in L^{p}(\mathcal{F}_t)} \gamma(s,t,u,X)$ of $\alpha_{su}(Q)$. The term $\esssup_{X \in L^{p}(\mathcal{F}_t)} \gamma(s,t,u,X)$ can be then interpreted as the maximal error (at the level of the penalty function) for a wrong use of $\rho_{su}$ for positions $X$ belonging to $L^{p}(\mathcal{F}_t)$, that is, $\rho_{s \cdot}$ with a wrong time horizon.
\medskip

From a motivational point of view, one can imagine $\gamma$ dependent on the time horizon, that is, $\gamma (s,t,u,X)=\gamma_{s,t}(h,X)$ with $h=u-t$ or, even, $\gamma_{s,t}(h)$ independent from $X$. In particular, $h=u-t$ can be interpreted as the length of the time interval over which there is an uncorrect use of the risk measure ($\rho_{su}$ versus $\rho_{st}$). See Example~\ref{ex: longevity-bsde} and Example~\ref{ex: longevity-bsvie}.

%
%

\smallskip
In what follows, we investigate whether BSDEs or BSVIEs can provide families of fully-dynamic risk measures satisfying sub or weak time-consistency and h-longevity. In particular, we study the cases of a single driver $g$ and of a family of drivers $(g_t)_{t \in [0,T]}$ depending on the time horizon $t$ considered in $\rho_{st}$. The former case will be shortened by BSDE $(g)$ (or BSVIE $(g)$), the latter by BSDE $(g_t)$ (or BSVIE $(g_t)$).

\section{Relation with BSDEs} \label{sec: bsdes}

Hereafter, we consider a fully-dynamic risk measure induced by a single BSDE with driver $g$ and also measures generated by a family of BSDEs associated to the drivers $\mathcal{G}=(g_t)_{t \in [0,T]}$ where the index $t$ refers to the time horizon.
We will restrict our attention to $L^2$ spaces.\medskip

On the probability space $(\Omega, \mathcal{F},P)$ we consider a $d$-dimensional Brownian motion $(B_t)_{t \in [0,T]}$ and the $P$-augmented natural filtration $(\mathcal{F}_t)_{t \in [0,T]}$ of $(B_t)_{t \in [0,T]}$.
According to Peng~\cite{peng97}, the solution $(Y_t,Z_t)_{t \in [0,T]}$ of the BSDE
\begin{equation} \label{eq: BSDE}
Y_t= X+ \int_t^T g(s,Y_s,Z_s) \, ds - \int_t^T Z_s \, dB_s
\end{equation}
can be seen as an operator depending on the driver $g$ and evaluated at the final condition $X \in L^2( \mathcal{F}_T)$.
In our work, we consider a driver
\begin{equation*}
g: \Omega \times [0,T]\times \Bbb R \times \Bbb R^d \to \Bbb R
\end{equation*}
satisfying the \textit{standard assumptions}:
\begin{itemize}
\item adapted,
\item uniformly Lipschitz, i.e. there exists a constant $C>0$ such that, $dP \times dt$-a.e.,
\begin{equation*}
\vert g(\omega,t, y_1, z_1) - g(\omega, t , y_2, z_2)\vert \leq C (\vert y_1 - y_2 \vert + \vert z_1 - z_2\vert ),
\end{equation*}
for any $y_1, y_2 \in \Bbb R, \, z_1, z_2 \in \Bbb R^d$,
where $\vert \cdot \vert$ denotes the Euclidean norm in $\Bbb R^k$ for whatever $k$ is relevant;
\item $E\left[\int_0^T \vert g(s,0,0)\vert ^2 \, ds \right]<+\infty$.
\end{itemize}
Under these conditions, equation~\eqref{eq: BSDE} admits a unique solution $(Y_t,Z_t)_{t \in [0,T]}$, with $(Y_t)_{t \in [0,T]} \in \Bbb H^2_{[0,T]} (\Bbb R)$ and $(Z_t)_{t \in [0,T]} \in \Bbb H^2_{[0,T]} (\Bbb R^d)$, where we have set
\begin{equation*}
\Bbb H^2_{[a,b]} (\Bbb R^k) \hspace{-1mm} \triangleq \hspace{-1mm} \Big\{
\mbox{adapted } \Bbb R^k\mbox{-valued processes } (\eta_s)_{s \in [a,b]}: \hspace{-1mm} E\Big[\hspace{-1mm}\int_a^b \hspace{-2mm} \vert \eta_s \vert^2 \, ds \Big] \hspace{-1mm} < \hspace{-1mm} \infty
\Big\}.
\end{equation*}
For further details see, e.g., El Karoui et al.~\cite{EKPQ}.

In Peng~\cite{peng97,peng05} and Rosazza Gianin~\cite{rg} the relationship among BSDEs, nonlinear expectations and dynamic risk measures is detailed clarifying that the properties of $g$ reflect the properties of the risk measures associated. To summarize, in the financial context of monetary risk measures, if $g$ in~\eqref{eq: BSDE} does not depend on $y$, then the risk measure is translation invariant.
If $g$ depends monotonically on $y$ then the risk measure is cash-subadditive, as discussed in El Karoui and Ravanelli~\cite{EK-rav}.
Also $g$ convex produces a convex risk measure. See Barrieu and El Karoui~\cite{barrieu-el-karoui}, Jiang~\cite{jiang}, and Rosazza Gianin~\cite{rg}.

Beyond the Lipschitz condition on the driver, the connection between BSDEs and risk measures has been studied in terms of maximal solutions of BSDEs (see, e.g., Barrieu and El Karoui~\cite{barrieu-el-karoui} and Kobylanski~\cite{kolylanski}). When it comes to generalizations beyond the Brownian framework we can refer, e.g., to Royer~\cite{royer}, Quenez and Sulem~\cite{quenez-sulem}, and Laeven and Stadje~\cite{laeven-stadje}.

\subsection{Risk measures generated by a single BSDE} \label{sec: BSDEg}

In this work we have assumed cash-invariant risk measures, then we consider
the following BSDE with driver $g$, not depending on $y$, and terminal condition $X \in L^2 (\mathcal{F}_u)$:
\begin{equation} \label{eq: BSDEg}
Y_t= X + \int_t^u g(s,Z_s) ds - \int_t^u Z_s dB_s, \quad 0 \leq t \leq u,
\end{equation}
where $u \in [0,T]$.
To simplify the notation, we will occasionally adopt the Peng's notation where $\mathcal{E}^g \left(X \vert \mathcal{F}_t \right)$ denotes the conditional $g$-expectation of $X$ at time $t$, that is the $Y$-component of the solution $(Y,Z)$ at time $t$ of the BSDE above.Then, we focus on risk measures of the following form:
\begin{equation} \label{eq: rho_BSDEg}
\rho_{tu} (X)= \mathcal{E}^g \left(-X \vert \mathcal{F}_t \right), \quad X \in L^{2} (\mathcal{F}_u).
\end{equation}

The condition $g(t,0)=0$, guarantees the restriction property as well as the normalization~\eqref{normalization}. Hence, in our case we cannot assume $g(t,0)=0$.

\begin{proposition} \label{prop: norm-restrict-BSDE-g}
The following statements are equivalent:

\noindent (a) $g(t,0)=0$ for any $t \in [0,T]$;

\noindent (b) $\rho$ is normalized, see~\eqref{normalization};

\noindent (c) $\rho$ has the restriction property~\eqref{restriction}.
\end{proposition}

\noindent
\begin{proof}
(a) $\Rightarrow$ (b). This implication follows immediately by remarking that $(Y_t,Z_t)$ with $Y_t=Z_t=0$ for any $t$ is the unique solution when the terminal condition is $0$. See also Peng~\cite{peng97}.\smallskip

\noindent
(b) $\Rightarrow$ (a). Assume that $\rho_{tu}(0)=0$ for any $0\leq t \leq u \leq T$. It follows that
$$
0=Y_t=  \int_t^u g(s,Z_s) ds - \int_t^u Z_s dB_s
$$
holds for any $0\leq t \leq u \leq T$. Hence
$$
\int_t^u g(s,Z_s) ds = \int_t^u Z_s dB_s, \quad \mbox{ for any } 0\leq t \leq u \leq T.
$$
Since a continuous martingale and a process of finite variation can be equal only if the martingale is constant (see Prop. 1.2 in Chapter 4 of Revuz and Yor~\cite{revuz-yor}), then for any $t$ the martingale $M_t (u)=\int_t^u Z_s dB_s$, $u \geq t$, starting from $0$ is identically equal to $0$. Hence $Z_s \equiv 0$ for any $s \geq t$ and $\int_t^u g(s,Z_s) ds \equiv 0$ for any $u \geq t$. By replacing $Z_s$ and deriving this last equation with respect to $u$, it follows that $g(u,0)=0$ for any $u$.

 \noindent
(a) $\Rightarrow$ (c). This implication was proved in Peng~\cite{peng97}. Assume that $X \in L^{2}(\mathcal{F}_u)$ and consider $\rho_{tu} (X)$ and $\rho_{tv} (X)$ for any $v \geq u$. Let us denote by $(Y^{X,u}_r,Z^{X,u}_r)$ (resp. $(Y^{X,v}_r,Z^{X,v}_r)$) the solution corresponding to $\rho_{tu} (X)$ (resp. $\rho_{tv} (X)$) at time $r \leq u$. Since $(Y^{X,v}_r,Z^{X,v}_r)$ with
$$
Y^{X,v}_r= \left\{
\begin{array}{rl}
Y^{X,u}_r;& r \leq u \\
-X;& u < r \leq v
\end{array}
\right. \quad
Z^{X,v} (r,s)= \left\{
\begin{array}{rl}
Z^{X,u}_r;& s \leq u \\
0;& u < s \leq v
\end{array}
\right.
$$
is a solution of $\rho_{tv} (X)$ when $g(t,0)=0$ for any $t \in [0,T]$, the restriction property follows.

\noindent
(c) $\Rightarrow$ (a). By translation invariance and restriction, it holds that for any $X \in L^2 (\mathcal{F}_t) \subseteq L^2 (\mathcal{F}_u)$
$$
-X+\rho_{tT}(0)=\rho_{tT}(X)=\rho_{tu}(X)=-X + \rho_{tu}(0).
$$
Hence $\rho_{tT}(0)=\rho_{tu}(0)$ for any $t \leq u \leq T$. Since $\rho_{tu}$ is induced by the BSDE~\eqref{eq: BSDEg} via~\eqref{eq: rho_BSDEg}, it follows that $\rho_{tt}(0)=0$. Consequently, $\rho_{tu}(0)=0$ for any $t \leq u \leq T$. By this the proof is complete.
\end{proof}

\bigskip
Note that while normalization refers to a single risk measure $\rho_{tu}$, restriction involves the whole family $(\rho_{tu})_{t,u}$.
It is already known that any BSDE satisfies both weak and strong time-consistency (see Barrieu and El Karoui~\cite{barrieu-el-karoui}, El Karoui et al.~\cite{EKPQ}, Bion Nadal and Di Nunno~\cite{bion-nadal-di-nunno}). From the arguments above, any BSDE with a driver $g$ such that $g(t,0)=0$ for any $t \in [0,T]$ satisfies normalization, restriction and strong time-consistency. Thus, it also trivially satisfies h-longevity. \medskip

Here below, we provide some examples where the driver is not normalized and $(\rho_{tu})_{t,u}$ does not satisfies restriction, but only (strict) h-longevity.

\begin{example} \label{ex: BSDE g non normalized}
a) Consider the driver $g(t,z)=a$ for any $t \in [0,T], z \in \mathbb{R}^d$, with $a \in \mathbb{R} \setminus \{0\}$.
It can be checked easily that, for any $t \in [0,T]$ and $X \in L^2(\mathcal{F}_T)$,
\begin{equation*}
\rho_{tT}(X)= E_P \left[\left. -X + \int_t ^T a \, ds\right| \mathcal{F}_t \right]=E_P \left[\left. -X \right| \mathcal{F}_t \right] + (T-t)a.
\end{equation*}
This means that, for $a \neq 0$, $(\rho_{tu})_{t,u}$ is not normalized and does not satisfy the restriction property. Instead, it satisfies h-longevity whenever $a >0$. \medskip

\noindent
b) Consider now the following driver: $g(t,z)=b z +a$ for any $t \in [0,T], z \in \mathbb{R}^d$, with $a, b \in \mathbb{R} \setminus \{0\}$.
By Girsanov Theorem,
\begin{eqnarray*}
-dY_t&= (b Z_t +a) dt- Z_t \, dB_t
&= a \, dt- Z_t \, dB_t^Q
\end{eqnarray*}
where
$E \Big[  \frac{d Q}{dP} \Big| \mathcal{F}_t \Big]= \exp \left\{ - \frac 12 b^2 t + b \cdot B_t \right\}$.
It then follows easily that
\begin{equation*}
\rho_{tT}(X)= E_Q \left[\left. -X \right| \mathcal{F}_t \right] + (T-t)\, a.
\end{equation*}
As before, for $a \neq 0$, $(\rho_{tu})_{t,u}$ is not normalized and does not satisfy the restriction property. Instead, it satisfies h-longevity whenever $a=g(t,0) >0$.
\end{example}

The next result provides a sufficient condition on the driver $g$ for h-longevity of $(\rho_{tu})_{t,u}$.

\begin{proposition} \label{prop: longevity-BSDE}
If $g(v,0) \geq 0$ for any $v \in [0,T]$, then h-longevity holds.
Furthermore, for all $s,u \in [0,T]$, $s \leq u$, $\gamma(s,t,u,X)=E_{\widetilde{Q}_X} \left[\int_t^u g(v,0) dv | \mathcal{F}_s\right]$, $s\leq  t\leq u, X\in L^p(\mathcal{F}_t)$,
where $\widetilde{Q}_X$ is a probability measure on $\mathcal{Q}_{su}$ depending on $X$ equivalent to $P$, with density
\begin{equation*}
\frac{d \widetilde{Q}_X}{dP}= \exp \left\{ - \frac 12 \int_s^u \vert\Delta_z g(v)\vert^2 dv + \int_s^u \Delta_z g(v) dB_v \right\}.
\end{equation*}
Here above $\Delta_z g(v) = (\Delta_z^i g(v))_{i=1,...,d}$ and
\begin{equation*}
\Delta_z^i g(v) \triangleq \frac{g(v,Z_v^u)- g(v,\bar{Z}_v^t)}{d(Z_v^{u,i}- \bar{Z}_v^{t,i})} 1_{\{Z_v^{u,i} \neq \bar{Z}_v^{t,i}\}}.
\end{equation*}
\end{proposition}
The probability measure ${\widetilde{Q}_X}$ here above can be interpreted as an {\it h-longevity premium measure}.

\smallskip
\noindent
\begin{proof}
Let $s \leq t \leq u$ and $X \in L^p (\mathcal{F}_t)$. The risk measures
$\rho_{st}(X)$ and $\rho_{su}(X)$ satisfy the following BSDEs:
\begin{eqnarray*}
\rho_{st}(X)&=& -X+ \int_s^t g(v,Z_v^t) dv - \int_s^t Z_v ^t dB_v \\
\rho_{su}(X)&=& -X+ \int_s^u g(v,Z_v^u) dv - \int_s^u Z_v ^u dB_v,
\end{eqnarray*}
respectively. Set now the $\Bbb R^d$-valued process
\begin{equation*}
\bar{Z}_v^t=\left\{
\begin{array}{rl}
Z_v^t;& v \leq t \\
0;& t<v \leq u
\end{array}
\right. ;
\quad \widetilde{Z}_v= Z_v^u - \bar{Z}_v^t.
\end{equation*}
Then
\begin{eqnarray}
&&\rho_{su}(X)-\rho_{st}(X) \notag\\
&&=  \int_s^u \hspace{-1mm}[g(v,Z_v^u)- g(v,\bar{Z}_v^t)] dv
+ \hspace{-1mm} \int_t^u \hspace{-1mm}g(v,\bar{Z}_v^t) dv
- \hspace{-1mm}\int_s^u \hspace{-1mm} [Z_v ^u- \bar{Z}_v^t] dB_v
- \hspace{-1mm}\int_t^u \hspace{-1mm} \bar{Z}_v ^t dB_v \notag \\
&&=  \int_s^u [g(v,Z_v^u)- g(v,\bar{Z}_v^t)] dv - \int_s^u \widetilde{Z}_v dB_v+ \int_t^u g(v,0) dv  \notag \\
&&=  \int_s^u \Delta_z g(v) \cdot \widetilde{Z}_v dv - \int_s^u \widetilde{Z}_v dB_v+ \int_t^u g(v,0) dv .\label{eq: bsde-longevity-1}
\end{eqnarray}
Furthermore,~\eqref{eq: bsde-longevity-1} can be rewritten also as
\begin{equation} \label{eq: linear BSDE-longevity}
\delta \rho_s= \Gamma^{t ,u} + \int_s^u \Delta_z g(v) \cdot \widetilde{Z}_v dv - \int_s^u \widetilde{Z}_v dB_v,
\end{equation}
where $\delta \rho_s \triangleq \rho_{su}(X)-\rho_{st}(X)$ and $\Gamma^{t,u} \triangleq \int_t^u g(v,0) dv$ represents the final condition at time $u$ (which depends on $t$ but not on $s$) of the linear BSDE~\eqref{eq: linear BSDE-longevity}.

Since $\Gamma^{t,u}\geq 0$ for any $t$ by hypothesis and $\Delta_z g(v) \in \Bbb H^2_{[s,u]} (\Bbb R^d)$ by the assumption of $g$ Lipschitz in $z$, by Prop. 2.2 of El Karoui, Peng and Quenez~\cite{EKPQ} it follows that $\delta \rho_s \geq 0$ for any $s \leq t$.

By applying Girsanov Theorem,~\eqref{eq: bsde-longevity-1} becomes
\begin{eqnarray*}
\rho_{su}(X)-\rho_{st}(X)
&=&  \int_s^u \Delta_z g(v) \cdot \widetilde{Z}_v dv - \int_s^u \widetilde{Z}_v dB_v+ \int_t^u g(v,0) dv \\
&=& - \int_s^u \widetilde{Z}_v dB^{\widetilde{Q}_X}_v+ \int_t^u g(v,0) dv,
\end{eqnarray*}
where $B^{\widetilde{Q}_X}_v \triangleq B_v - B_s - \int_{s}^v \Delta_z g(r) \, dr$, $v \in [s,u]$, is a $\widetilde{Q}_X$-Brownian motion. Hence, by taking the conditional expectation with respect to $\widetilde{Q}_X$,
\begin{eqnarray*}
\rho_{su}(X)-\rho_{st}(X)
&=&  E_{\widetilde{Q}_X} \left[ - \int_s^u \widetilde{Z}_v dB^{\widetilde{Q}_X}_v+ \int_t^u g(v,0) dv \Big\vert \mathcal{F}_s\right]\\
&=&E_{\widetilde{Q}_X} \left[\int_t^u g(v,0) dv \Big\vert \mathcal{F}_s\right].
\end{eqnarray*}
By assumption on $g(\cdot, 0)$, it follows that $\rho_{su}(X)-\rho_{st}(X) \geq 0$ and that $\gamma(s,t,u,X)=\rho_{su}(X)-\rho_{st}(X)=E_{\widetilde{Q}_X} \left[\int_t^u g(v,0) dv | \mathcal{F}_s\right]$.
\end{proof}
\medskip

As discussed in Section~\ref{sec: longevity}, $\gamma$ may depend on the time horizon, that is, $\gamma (s,t,u,X)=\gamma_{s,t}(h,X)$ with $h=u-t$ or, even, $\gamma_{s,t}(h)$ independent from $X$. The following example provides some cases covering the situation above.

\begin{example} \label{ex: longevity-bsde}
Let $g(v,0) \geq 0$ for any $v \in [0,T]$. Hence, by the result above, h-longevity holds and $\gamma(s,t,u,X)=E_{\widetilde{Q}_X} \left[\int_t^u g(v,0) dv | \mathcal{F}_s\right]$ for any $X \in L^2(\mathcal{F}_t)$.\smallskip

\noindent
a) If $g(v,0)=c$ for any $v \in [0,T]$, with $c\geq 0$, then $c$ is necessarily deterministic (since it should be measurable for any $v \geq 0$) and, consequently,
\begin{equation*}
\gamma(s,t,u,X)=E_{\widetilde{Q}_X} \left[\int_t^u g(v,0) dv \Big\vert \mathcal{F}_s\right]=(u-t)c.
\end{equation*}
 Hence $\gamma$ only depends on $h=u-t$, that is, roughly speaking, on the length of the time interval over which there is an uncorrect use of the risk measure ($\rho_{su}$ versus $\rho_{st}$).\smallskip

\noindent
b) If $g(v,0)=\exp(-r \, v)$ for any $v \in [0,T]$, with $r\geq 0$, then $r$ is necessarily deterministic (for the same arguments as above) and, consequently,
\begin{equation*}
\gamma(s,t,u,X)=\frac{e^{-rt} \left( 1-e^{-r (u-t)} \right)}{r}
\end{equation*}
In other words, $\gamma$ only depends on the ``right'' time horizon $t$ (referring to the measurability of $X$) and on the length of the time interval $[t,u]$.
\end{example}

\subsection{Risk measures generated by a family of BSDEs}

Now we consider general risk measures induced by a family of BSDEs of type~\eqref{eq: BSDE} with drivers $\mathcal{G}=(g_u)_{u \in [0,T]}$ depending on the time horizon $u$ of $\rho_{tu}$.
While the focus of this paper is to work with BSDEs where the driver is independent on $y$ (leading to cash-invariant risk measures), we still introduce some concepts and results for the general case which serves the cash-subadditive case of El Karoui and Ravanelli \cite{EK-rav}.

For later use, by \textit{increasing family} $\mathcal{G}=(g_u)_{u \in [0,T]}$ of BSDE drivers it is meant, for any $t \leq u$,
\begin{equation*} 
g_t(v,y,z) \leq g_u(v,y,z)  \mbox{ for any } v \in [0,t], y \in \Bbb R, z \in \Bbb R^d.
\end{equation*}

Suppose that, for any $t \leq u$, the risk measure $\rho_{tu}$ comes from a $g_u$-expectation with a driver depending on the maturity $u$. This means that
\begin{equation} \label{eq: rho-from-bsde-gt}
\rho_{tu}(X)=\rho_{tu}^{\mathcal{G}}(X)=\mathcal{E}^{g_u} (-X \vert \mathcal{F}_t), \mbox{ for any } X \in L^{2} (\mathcal{F}_u).
\end{equation}

Assume now that $(g_u)_{u \in [0,T]}$ is a family of drivers depending on the maturity $u$, independent of $y$, Lipschitz, and convex in $z$. Then the risk measure $(\rho^{\mathcal{G}}_{tu})_{t,u}$ defined by~\eqref{eq: BSDEg},~\eqref{eq: rho-from-bsde-gt} satisfies monotonicity, convexity, continuity from above/below, and translation invariance. Furthermore, for $u \in [0,T]$, if $g_u(v,0)=0$ for any $v \leq u$, then $\rho^{\mathcal{G}}_{tu} (0)=0$ for any $t \leq u$. In general, however, $g_u(v,0)=0$ for any $v \leq u$ \textit{does not} imply the restriction property.
\bigskip

In a Brownian setting, $\rho_{tu}^{\mathcal{G}}$ can be represented as
\begin{equation*} 
\rho_{tu}(X)=\rho_{tu}^{\mathcal{G}}(X)=\underset{Q \in \mathcal{Q}_{tu}}{\esssup} \, \left\{ E_Q \left[-X \vert \mathcal{F}_t \right] -\alpha_{tu}^{\mathcal{G}} (Q) \right\}, \quad X \in L^{2} (\mathcal{F}_u),
\end{equation*}
where $\mathcal{Q}_{tu}$ is defined in~\eqref{eq: Q_tu - gt}
and the penalty term $\alpha_{tu}^{\mathcal{G}}$ associated to $\rho_{tu}^{\mathcal{G}}$ is given by
\begin{equation*}
\alpha_{tu}^{\mathcal{G}}(Q)=E_Q \left[\left. \int_t ^u g^*_u (v,q_v) dv \right| \mathcal{F}_t \right],
\end{equation*}
where $g^*_u$ denotes the convex conjugate of the function $g_u$. See Delbaen et al.~\cite{DPR}.\bigskip

The following result shows that, for $g_u(r,0)=0$ for any $r, u$, the restriction property is satisfied only for risk measures induced by a BSDE with a single driver $g$ (that is, constant with respect to the maturity time).

\begin{proposition} \label{prop: restriction-BSDE-family}
For any $u \in [0,T]$, let $g_u(v,0)=0$ for any $v \leq u$.

The restriction property~\eqref{restriction} holds if and only if $g_u$ is constant in $u$.
\end{proposition}

\noindent
\begin{proof}
If $g_u$ is constant in $u$, the restriction property follows directly by Proposition~\ref{prop: norm-restrict-BSDE-g}.

Conversely, assume that the restriction property holds, i.e. $\rho_{tu}(X)=\rho_{tv}(X)$ for any $t \leq u \leq v$ and $X \in L^2(\mathcal{F}_u)$.
Proceeding as in the proof of the Converse Comparison Theorem of Briand et al.~\cite{BCHMPeng}, Thm. 4.1, and Jiang~\cite{jiang}, Lemma 2.1,
\begin{eqnarray*}
g_u(s,z)=\lim _{\varepsilon \to 0} \frac{\rho_{su}(-z \cdot (B_{s + \varepsilon} -B_{\varepsilon}))}{\varepsilon}\\
g_v(s,z)=\lim _{\varepsilon \to 0} \frac{\rho_{sv}(-z \cdot (B_{s + \varepsilon} -B_{\varepsilon}))}{\varepsilon}
\end{eqnarray*}
with convergence in $L^p$ with $p \in [1,2)$, for any $z \in \Bbb R^d$, $u \leq v$ and $s \in [0,u]$. By extracting a subsequence to obtain convergence $P$-a.s. and passing to the limit as $\varepsilon \to 0$, it holds that
\begin{equation*}
\frac{\rho_{sv}(-z (B_{s + \varepsilon} -B_{\varepsilon}))}{\varepsilon}=\frac{\rho_{su}(-z (B_{s + \varepsilon} -B_{\varepsilon}))}{\varepsilon} \longrightarrow g_u(s,z), \quad
\epsilon \to 0, \quad  P-a.s.
\end{equation*}
where the equality is due to restriction. The thesis then follows because
\begin{equation*}
\frac{\rho_{sv}(-z (B_{s + \varepsilon} -B_{\varepsilon}))}{\varepsilon} \longrightarrow  g_v(s,z),  \quad
\epsilon \to 0, \quad  P-a.s.\end{equation*}
By this we end the proof.
\end{proof}

\bigskip
As done in Example~\ref{ex: BSDE g non normalized} for a single driver, we provide here below some examples where restriction fails.

\begin{example} \label{ex: BSDE gu non normalized}
a) Consider the driver $g_u(t,z)=a_u$ for any $t \in [0,u], z \in \mathbb{R}^d$, with $a_u \in \mathbb{R} \setminus \{0\}$ depending on the maturity $u$.
It can be checked easily that, for any $t \in [0,T]$ and $X \in L^2(\mathcal{F}_T)$,
\begin{equation*}
\rho_{tu}(X)= E_P \left[\left. -X \right| \mathcal{F}_t \right] + (u-t)a_u.
\end{equation*}
This means that, for $a_u \neq 0$, $(\rho_{tu})_{t,u}$ is not normalized and does not satisfy the restriction property. Instead, it satisfies h-longevity whenever $a_u >0$ is increasing in $u$. \medskip

\noindent
b) Consider now the following driver: $g_u(t,z)=b z +a_u$ for any $t \in [0,u], z \in \mathbb{R}^d$, with $a_u, b \in \mathbb{R} \setminus \{0\}$ and $a_u$ depending on the maturity $u$.
Similarly to Example~\ref{ex: BSDE g non normalized}, it follows that
\begin{equation*}
\rho_{tu}(X)= E_Q \left[\left. -X \right| \mathcal{F}_t \right] + (u-t)a_u,
\end{equation*}
where $E \left[ \left. \frac{d Q}{dP} \right| \mathcal{F}_t \right]= \exp \left\{ - \frac 12 b^2 t + b \cdot B_t \right\}$.
As before, therefore, for $a_u \neq 0$, $(\rho_{tu})_{t,u}$ is not normalized and does not satisfy the restriction property. Instead, it satisfies h-longevity whenever $a_u >0$ is increasing in $u$.
\end{example}

\subsubsection{Time-consistency}

Thanks to the Comparison Theorem (see El Karoui et al.~\cite{EKPQ}) and to the Converse Comparison Theorem (see Briand et al.~\cite{BCHMPeng}), the following result establishes that monotonicity of the family $\mathcal{G}=(g_t)_{t \in [0,T]}$ is equivalent to sub time-consistency.

\begin{theorem} \label{prop: sub-tc-bsde-gt}
Let $(\rho_{tu})_{t,u}$ be induced by a BSDE with a family of drivers $\mathcal{G}=(g_t)_{t \in [0,T]}$ as in~\eqref{eq: rho-from-bsde-gt}.

\noindent
a) The family $\mathcal{G}$ is increasing if and only if $(\rho_{tu})_{t,u}$ satisfies sub time-consistency.

\noindent
b) $\mathcal{G}= \{ g \}$  if and only if $(\rho_{tu})_{t,u}$ satisfies strong time-consistency.
\end{theorem}

\noindent
\begin{proof}
a) By definition~\eqref{eq: rho-from-bsde-gt}, sub time-consistency of $(\rho_{tu})_{t,u}$ can be written as
\begin{equation} \label{eq: sub tc-proof}
\mathcal{E}^{g_t} (\mathcal{E}^{g_u} (-X \vert \mathcal{F}_t) \vert \mathcal{F}_s) \leq \mathcal{E}^{g_u} (-X \vert \mathcal{F}_s)
\end{equation}
for any $s \leq t \leq u$ and $X \in L^2(\mathcal{F}_u)$.

By definition of $\mathcal{E}^{g_t}$, the left-hand and right-hand sides of the previous equation can be rewritten as follows:
\begin{eqnarray*}
&&\mathcal{E}^{g_t} (\mathcal{E}^{g_u} (-X \vert \mathcal{F}_t) \vert \mathcal{F}_s) \\
=&& \mathcal{E}^{g_u} (-X \vert \mathcal{F}_t)+ \int_s ^t g_t(v,Z_v) \, dv - \int_s ^t Z_v \, dB_v \\
=&& -X +\int_t ^u g_u(v,\tilde{Z}_v) \, dv - \int_t^u \tilde{Z}_v \, dB_v +\int_s ^t g_t(v,Z_v) \, dv -\int_s ^t  Z_v \, dB_v \\
\end{eqnarray*}
and
\begin{equation}  \label{eq: sub tc-proof-2}
\mathcal{E}^{g_u} (-X \vert \mathcal{F}_s)  = -X +\int_s ^u g_u(v,\hat{Z}_v) \, dv - \int_s^u \hat{Z}_v \, dB_v.
\end{equation}
Set now
\begin{equation*}
\bar{g}(v,z)= \left\{
\begin{array}{rl}
g_t(v,z);& \quad v \in [0,t] \\
g_u(v,z);& \quad v \in (t,u]
\end{array}
\right.
\quad
\bar{Z}_v= \left\{
\begin{array}{rl}
Z_v;& \quad v \in [0,t] \\
\tilde{Z}_v;& \quad v \in (t,u]
\end{array}
\right.
\end{equation*}
By the arguments above, the left-hand side of~\eqref{eq: sub tc-proof} becomes
\begin{equation} \label{eq: sub tc-proof-3}
\mathcal{E}^{g_t} (\mathcal{E}^{g_u} (-X \vert \mathcal{F}_t) \vert \mathcal{F}_s) = -X +\int_s ^u \bar{g}(v,\bar{Z}_v) \, dv - \int_s^u \bar{Z}_v \, dB_v .
\end{equation}
On the one hand, if $\bar{g} \leq g_u$ on $[0,u]$ (or, equivalently, $g_t \leq g_u$ on $[0,t]$), then by the Comparison theorem of BSDE (see El Karoui et al.~\cite{EKPQ}) sub time-consistency~\eqref{eq: sub tc-proof} is satisfied.

On the other hand, by~\eqref{eq: sub tc-proof-2} and~\eqref{eq: sub tc-proof-3} and Converse Comparison Theorem of Briand et al.~\cite{BCHMPeng}, sub time-consistency~\eqref{eq: sub tc-proof} implies that $\bar{g} \leq g_u$ on $[0,u]$, hence $g_t \leq g_u$ on $[0,t]$.\smallskip

\noindent
b) The case of time-consistency can be obtained by replacing inequalities with equalities in the proof above.
\end{proof}

\smallskip
\noindent
In other words, the result above guarantees that strong time-consistency is fulfilled if and only if the family of drivers reduces to a singleton, that is, the drivers $g_t$ do not depend on the maturity $t$.

\begin{remark}
 In the previous result, the proof of item a) can be done also by means of the penalty term characterization. In this case, one could use
 \begin{equation*}
 \alpha_{su}^{\mathcal{G}}(Q) = E_Q \left[\left. \int_s ^u g^*_u (v,q_v) dv \right| \mathcal{F}_s \right]
 \end{equation*}
and the fact that, by the Fenchel-Moreau biconjugate Theorem, the condition $g^* _t \geq g^* _u$ is equivalent to $g _t \leq g _u$. Indeed, if $g^* _t \geq g^* _u$ then
\begin{equation*}
g_t(s,z)= \sup_{ q \in \Bbb R^d} \{q \cdot z- g_t ^*(s,q) \} \leq \sup_{ q \in \Bbb R^d } \{q \cdot z- g_u ^*(s,q) \}= g_u(s,z).
\end{equation*}
The converse implication can be checked similarly.
\end{remark}

\subsubsection{H-longevity and families of BSDEs}
In order to investigate h-longevity, we need to compare BSDEs with different horizons. Hereafter, we suggest a comparison theorem for this. Note that in this result we consider more general BSDEs~\eqref{eq: BSDE} with driver depending on $y$, as the result is of interest well beyond the use of BSDEs in risk measures modelling.

Let $T_1, T_2>0$ be two different time horizon with $T_2 > T_1$. Let $(Y_t ^{T_i, \xi _i}, Z_t ^{T_i, \xi _i})_{t \in [0,T_1]}$ (for $i=1,2$) be the solution of a BSDE with final condition $\xi _i \in L^{2} (\mathcal{F}_{T_i})$ and driver $g^{T_i}$ satisfying the usual assumptions. More precisely,
\begin{equation} \label{eq: BSDE - different T}
Y_t ^{T_i}= \xi _i + \int_t ^{T_i} g^{T_i}(s, Y_s ^{T_i}, Z_s ^{T_i}) ds- \int_t ^{T_i} Z_s ^{T_i} dB_s.
\end{equation}
In the terminology of Peng~\cite{peng97}, $Y_t ^{T_i}= \mathcal{E}^{g^{T_i}} (\xi_i \big| \mathcal{F}_t)$.

\begin{theorem} \label{prop: extension-comparison_BSDE} (Horizon Comparison Theorem)
If $g^{T_2}(s,y,z) \geq g^{T_1}(s,y,z)$ for any $s \in [0,T_1], y \in \Bbb R,z \in \Bbb R^d$ and $g^{T_2}(s,y,z) \geq 0$ for any $s \in [T_1, T_2], y \in \Bbb R,z \in \Bbb R^d$ and $\xi_2 \geq \xi_1$, then $Y_t ^{T_2} \geq Y_t ^{T_1}$ for any $t \in [0,T_1]$ and $Y_t ^{T_2} \geq \xi_1$ for any $t \in [T_1, T_2]$.
\end{theorem}

\noindent
\begin{proof}
By~\eqref{eq: BSDE - different T}, it follows that
\begin{eqnarray}
Y_t ^{T_2} - Y_t ^{T_1} = && \xi _2 -\xi_1 + \int_t ^{T_2} g^{T_2}(s, Y_s ^{T_2}, Z_s ^{T_2}) ds - \int_t ^{T_1} g^{T_1}(s, Y_s ^{T_1}, Z_s ^{T_1}) ds \notag \\
&& - \int_t ^{T_2} Z_s ^{T_2} dB_s + \int_t ^{T_1} Z_s ^{T_1} dB_s .\label{eq: BSDE - different T-2}
\end{eqnarray}
Let us now extend $g^{T_1}(s,y , z)$, $Y_s ^{T_1}$ and $Z_s ^{T_1}$ also to $[T_1, T_2]$ as follows:
\begin{eqnarray*}
&&\overline{g}^{(T_1)}(s, y, z)= \left\{
\begin{array}{rl}
g^{T_1}(s, y, z);& s \in [0,T_1] \\
0;& s \in (T_1;T_2]
\end{array}
\right. \\
&&\overline{Y}_s ^{T_1}  = \left\{
\begin{array}{rl}
Y_s ^{T_1};& s \in [0,T_1] \\
\xi_1;& s \in (T_1;T_2]
\end{array}
\right.;  \qquad
\overline{Z}_s ^{T_1}  = \left\{
\begin{array}{rl}
Z_s ^{T_1};& s \in [0,T_1] \\
0;& s \in (T_1;T_2]
\end{array}
\right.
\end{eqnarray*}
For $t \in [0,T_2]$, equation~\eqref{eq: BSDE - different T-2} becomes then
\begin{equation} \label{eq: BSDE - different T-3}
Y_t ^{T_2} - \overline{Y}_t ^{T_1} = \xi _2 -\xi_1 + \int_t ^{T_2} \hspace{-2mm}\big[ g^{T_2}(s, Y_s ^{T_2}, Z_s ^{T_2}) - \overline{g}^{T_1}(s, \overline{Y}_s ^{T_1}, \overline{Z}_s ^{T_1}) \big]ds \\
- \int_t ^{T_2} \hspace{-1mm}(Z_s ^{T_2} - \overline{Z}_s ^{T_1} ) dB_s
\end{equation}
Set now $\widetilde{Y}_s \triangleq Y_s ^{T_2} - \overline{Y}_s ^{T_1}$, $\widetilde{\xi} \triangleq \xi_2 - \xi_1$, $\widetilde{Z}^i_s \triangleq Z_s ^{T_2, i} - \overline{Z}_s ^{T_1, i}$ for $i=1,...,d$, and
\begin{eqnarray*}
\Delta_s g &=& g^{T_2}(s, \overline{Y}_s ^{T_1}, \overline{Z}_s ^{T_1}) - \overline{g}^{T_1}(s, \overline{Y}_s ^{T_1}, \overline{Z}_s ^{T_1}) \\
\Delta_s y &=& \frac{g^{T_2}(s, Y_s ^{T_2}, Z_s ^{T_2}) - g^{T_2}(s, \overline{Y}_s ^{T_1}, Z_s ^{T_2})}{Y_s ^{T_2} - \overline{Y}_s ^{T_1}} 1_{\{Y_s ^{T_2} \neq \overline{Y}_s ^{T_1} \}} \\
(\Delta_s z)^i &=&  \frac{g^{T_2}(s, \overline{Y}_s ^{T_1}, Z_s ^{T_2}) - g^{T_2}(s, \overline{Y}_s ^{T_1}, \overline{Z}_s ^{T_1})}{d( Z_s ^{T_2, i} - \overline{Z}_s ^{T_1, i})} 1_{\{Z_s ^{T_2, i} \neq \overline{Z}_s ^{T_1, i} \}}
\end{eqnarray*}
for $i=1,...,d$. Hence~\eqref{eq: BSDE - different T-3} can be rewritten as
\begin{equation} \label{eq: BSDE - different T-4}
\left\{
\begin{array}{rl}
- d \widetilde{Y}_t  &=  \left[\Delta_t y \, \widetilde{Y}_t + \Delta_t z \cdot \widetilde{Z}_t + \Delta_t \, g \right] dt- \widetilde{Z}_t dB_t \\
\widetilde{Y}_{T_2}  &= \widetilde{\xi}
\end{array}
\right.
\end{equation}
Since $\Delta_t y \in \Bbb H^2_{[0,T]} (\Bbb R), \Delta_t z \in \Bbb H^2_{[0,T]} (\Bbb R^d)$ and $\Delta_t \, g$ is in $\Bbb H^2_{[0,T]} (\Bbb R)$ by the assumption of $g^{T_i}$ Lipschitz, the solution of~\eqref{eq: BSDE - different T-4} exists uniquely by Prop. 2.2 of El Karoui, Peng and Quenez~\cite{EKPQ}. Moreover, since $g^{T_2}(s,y,z) \geq \overline{g}^{T_1}(s,y,z)$ for any $s \in [0,T_2], y,z $ and $\widetilde{\xi} \geq 0$, then $\widetilde{Y}_t \geq 0$ for any $t \in [0,T_2]$ which concludes the proof.
\end{proof}
\medskip

The previous result guarantees that for fully-dynamic risk measures induced by $g$-expectations as in~\eqref{eq: rho-from-bsde-gt}, h-longevity holds when $g_u$ is increasing in $u$ and $g_u \geq 0$.

\begin{proposition} \label{prop: longevity-BSDE-gt}
a) If $\mathcal{G}$ is an increasing family of drivers and $\rho_{tu} (0) \geq 0$ for any $t \leq u$, then $(\rho_{tu})_{t,u}$ satisfies h-longevity.

\noindent
b) If $\mathcal{G}$ is an increasing family of drivers and $g_t \geq 0$ for any $t \in [0,T]$, then $(\rho_{tu})_{t,u}$ satisfies h-longevity.
\end{proposition}

\noindent
\begin{proof}
a) follows by Remark~\ref{rem: tc and longevity} and Proposition~\ref{prop: sub-tc-bsde-gt}.

\smallskip
\noindent
b) By Theorem~\ref{prop: extension-comparison_BSDE}, increasing monotonicity of the family $\mathcal{G}$ and $g_t \geq 0$ for any $t \in [0,T]$ imply h-longevity.
\end{proof}

\smallskip
Notice that requiring $g_u \geq 0$ implies that $\rho_{tu}(X) \geq E_P [- X \vert \mathcal{F}_t]$ for any $X \in L^2 (\mathcal{F}_u)$. Indeed, by definition of $\rho_{tu}(X) $ as in~\eqref{eq: rho-from-bsde-gt} and from $g_u \geq 0$, it follows that
\begin{eqnarray*}
\rho_{tu}(X) &=& E_P \left[\left. -X +\int_t ^u g_u(v,Z_v) \, dv - \int_t^u Z_v \, dB_v \right| \mathcal{F}_t \right] \\
&=& E_P \left[ -X   | \mathcal{F}_t \right] + E_P \left[\left. \int_t ^u g_u(v,Z_v) \, dv  \right| \mathcal{F}_t \right] \\
& \geq & E_P \left[ -X   | \mathcal{F}_t \right].
\end{eqnarray*}

\subsubsection{Beyond the Lipschitz driver: the entropic case} \label{sec: entrpic BSDE}

Next, we consider the entropic risk measure, which is generated by a quadratic BSDE. Even though the Lipschitz condition is clearly not satisfied in this case, the solution of the BSDE is still unique. This discussion on the uniqueness of solutions is developed in~\cite{barrieu-el-karoui}. The entropic risk measure is particularly interesting because of its explicit connection with utility functions as investment preference criteria.
In the following example, we propose a variation to the classical entropic risk measure able to capture h-longevity.

\begin{example} \label{ex: BSDE g non normalized - entropic}
In the one-dimensional case, for every $u\in [0,T]$, consider the driver $g_u(t,z)=\frac{b_u}{2} z^2 +a_u(t)$ for any $t \in [0,T], z \in \mathbb{R}$, with $b_u>0$ and $a_u$ positive function.
We recall that $b_u$ is the reciprocal of the risk aversion parameter in the corresponding utility function.
By using the explicit solution of the BSDE with driver $\bar g_u(t,z)=\frac{b_u}{2} z^2$, it can be checked easily that, for any $t \in [0,T]$ and $X \in L^2(\mathcal{F}_u)$,
\begin{equation*}
\rho_{su}(X)= \frac{1}{b_u} \ln \left(E_P \left[\left. \exp ( -b_uX) \right| \mathcal{F}_s \right]\right) + \int_s^u a_u(t)dt.
\end{equation*}
Hence, $(\rho_{su})_{s,u}$ is not normalized and does not satisfy the restriction property. Instead, {\color{blue} (strict)} h-longevity holds whenever $(b_u)_u$ and $(a_u)_u $ are increasing in $u$. See Proposition~\ref{prop: longevity-BSDE-gt}.
\end{example}

\section{Relation with BSVIEs} \label{sec: bsvies}

We have come now to consider a fully-dynamic risk measure induced by a single BSVIE with a driver $g$ and also measures generated by a family of BSVIEs associated to the drivers $\mathcal{G}=(g_t)_{t \in [0,T]}$ where the index $t$ refers to the time horizon.

\subsection{Risk measures generated by a single BSVIE}

In the same spirit of risk measures generated by BSDEs, we can consider risk measures induced by BSVIEs as Agram~\cite{agram} and Yong~\cite{yong07} suggested. Indeed, this can also be seen in connection with families of BSDEs as Yong~\cite{yong13-survey} has pinpointed.

Let us consider a BSVIE of type
\begin{equation} \label{eq: BSVIE}
Y_t= \xi_t + \int_t ^T g(t,s,Z(t,s)) \, ds - \int _t^T Z(t,s) \, dB_s,
\end{equation}
where $\xi_t\in L^2(\mathcal{F}_T)$, for all $t\in [0,T]$. The driver $g: \Omega \times \Delta \times \Bbb R^d \to \Bbb R$, with $\Delta_{[0,T]} \triangleq \{(t,s) \in [0,T] \times [0,T]: s\geq t \}$, is independent of $y$ (so to guarantee translation invariance in the application to monetary risk measures) and it is such that
\begin{itemize}
\item for any $t \in [0,T], z \in \Bbb R^d$, $g(t, \cdot, z)$ is adapted;
\item is uniformly Lipschitz, i.e. there exists a constant $C >0$ such that
\begin{equation*}
\vert g(t,s,z_1) - g(t,s,z_2) \vert \leq C \vert z_1 - z_2\vert, \mbox{ for any } (t,s) \in \Delta, z_1, z_2 \in \Bbb R^d, P-a.s.;
\end{equation*}
\item $E \left[\int_0^T \left(\int_t^T g(t,s,0) \, ds \right)^2 \, dt \right]<+\infty$.
\end{itemize}
Under there assumptions, we have a unique solution $(Y_{\cdot},Z(\cdot, \cdot)) \in \Bbb H^2_{[0,T]} (\Bbb R) \times \mathcal{Z}^2_{[0,T]} (\Bbb R^d)$, where we have set
\begin{equation*}
\mathcal{Z}^2_{[a,b]} (\Bbb R^d) \triangleq \left\{Z: \Omega \times \Delta_{[a,b]} \to \Bbb R^d: \begin{array}{rl}
&Z(t, \cdot) \mbox{ is adapted for all } t \mbox{ and }\\
&E \left[\int_a^b \int_a^b \vert Z(t,s)\vert ^2 \, ds \, dt \right]<+\infty
\end{array}
 \right\}.
\end{equation*}
See Yong~\cite{yong13-survey}, Thm. 2.2, and~\cite{yong07}.
Also Yong~\cite{yong13-survey} introduced the following family of BSDEs parameterized by $t$ and related to the BSVIE~\eqref{eq: BSVIE}:
\begin{equation*} \label{eq: BSVIE-parametrized-BSDE}
\eta (r;t, \xi_t)= \xi _t + \int_r^T \bar{g}(v,\zeta(v;t);t) dv - \int_r^T \zeta(v;t) dB_v, \quad r \in [t,T]
\end{equation*}
where
\begin{equation*}
\zeta(v;t)=Z(t,v); \quad \bar{g}(v,\zeta(v;t);t)=g(t,v,Z(t,v))\quad \mbox{and} \quad Y_t= \eta(t;t, \xi_t).
\end{equation*}
Although $\eta(\cdot;t,X)$ and $Y_t$ are closely related, $Y_t$, solution of the BSVIE~\eqref{eq: BSVIE} may satisfy some specific properties, since it corresponds not to the whole family $\eta(\cdot;t,X)$, but only to $\eta(t;t,X)$.

\smallskip
As in Section~\ref{sec: bsdes}, we assume convex drivers that guarantee convex risk measures satisfying monotonicity and translation invariance. See Yong~\cite{yong07} and Agram~\cite{agram}, who also makes an extension to jump dynamics.
Then we focus on monetary risk measures of the type
\begin{equation} \label{eq: rho_BSVIEg}
\rho_{tu} (X)= \mathcal{E}^{g,V} \left(-X \vert \mathcal{F}_t \right), \quad X \in L^{2} (\mathcal{F}_u),
\end{equation}
where $\mathcal{E}^{g,V} \left(-X \vert \mathcal{F}_t \right)$ denotes the $Y$-component of the solution  of the BSVIE of type~\eqref{eq: BSVIE}, with terminal condition $\xi_t= -X$, for all $t$.

As shown in the following result, for a driver independent on $y$ the condition $g(t,u,0)=0$, for any $t \leq u$, guarantees the restriction property as well as the normalization. Hence, in our case we cannot assume $g(t,u,0)=0$.

\begin{proposition} \label{prop: norm-restrict-BSVIE-g}
The following statements are equivalent:

\noindent (a) $g(t,u,0)=0$ for any $0 \leq t \leq u \leq T$;

\noindent (b) $\rho$ is normalized, see~\eqref{normalization}:

\noindent (c) $\rho$ satisfies the restriction property~\eqref{restriction}.
\end{proposition}

\noindent
\begin{proof}
(a) $\Rightarrow$ (b) The implication is true because $(Y_\cdot, ,Z(\cdot,\cdot))$ with $Y_t=Z(t,s)=0$,
$s \geq t$ is the unique solution corresponding to $X=0$.

\noindent
(b) $\Rightarrow$ (a). Assume that $\rho_{tu}(0)=0$ for any $0\leq t \leq u \leq T$. we have
$$
0=Y_t= \int_t^u g(t,s,Z(t,s)) ds - \int_t^u Z(t,s) dB_s
$$
holds for any $0\leq t \leq u \leq T$. Hence
$$
\int_t^u g(t,s,Z(t,s)) ds = \int_t^u Z(t,s) dB_s, \quad \mbox{ for any } 0\leq t \leq u \leq T.
$$
As in the proof of Proposition~\ref{prop: norm-restrict-BSDE-g}, it follows that for any fixed $t$ the martingale $M_t (u)=\int_t^u Z(t,s) dB_s$, with $u \geq t$, starting from $0$ should be identically to $0$. Hence $Z(t,s) \equiv 0$ for any $s \geq t$ and $\int_t^u g(t,s,0) ds \equiv 0$ for any $u \geq t$. By deriving this last equation with respect to $u$, it follows that $g(t,u,0)=0$ for any $u \geq t$.

\noindent
(a) $\Rightarrow$ (c). Assume that $X \in L^{2}(\mathcal{F}_u)$ and consider $\rho_{tu} (X)$ and $\rho_{tv} (X)$ for any $v \geq u$. Let us denote by $(Y^{X,u}_r,Z^{X,u}(r,s))$ (resp. $(Y^{X,v}_r,Z^{X,v}(r,s))$) the solution corresponding to $\rho_{tu} (X)$ (resp. $\rho_{tv} (X)$) at time $r \leq u$. Since $(Y^{X,v}_r,Z^{X,v}(r,s))$ with
$$
Y^{X,v}_r= \left\{
\begin{array}{rl}
Y^{X,u}_r;& r \leq u \\
-X;& u < r \leq v
\end{array}
\right. \quad
Z^{X,v} (r,s)= \left\{
\begin{array}{rl}
Z^{X,u}(r,s);& s \leq u \\
0;& u < s \leq v
\end{array}
\right.
$$
is a solution of $\rho_{tv} (X)$ whether $g(t,u,0)=0$ for any $0 \leq t \leq u \leq T$, the restriction property follows.

\noindent
(c) $\Rightarrow$ (a) The argument is similar to the proof of the corresponding implication in
Proposition~\ref{prop: norm-restrict-BSDE-g}.
\end{proof}

\smallskip
Note that, for all $t$,  the set of probability measures $Q_t$ on $\mathcal{F}_u$ such that
\begin{equation} \label{eq: definition-Qt}
\frac{dQ_t}{dP} \triangleq \exp \left\{\frac{1}{2} \int_t ^u \vert q(t,s) \vert ^2 ds- \int_t ^u q(t,s) dB_s \right\},
\end{equation}
for some stochastic process $(q(t,s))_{t \leq s \leq u} \in \Bbb H^2_{[t,u]}(\Bbb R^d)$ coincides with $\mathcal{Q}_{tu}$ in a Brownian setting (see Revuz and Yor~\cite{revuz-yor}, Chapter VIII, Prop. 1.6).

\begin{theorem}(Dual representation) \label{prop: dual-repres-bsvie}
For any $t \leq u$ fixed, the risk measure $\rho_{tu}$ induced by a BSVIE as in~\eqref{eq: rho_BSVIEg} has the following dual representation
\begin{equation} \label{eq: dual-repr-bsvie}
\rho_{tu} (X)= \underset{Q_t \in \mathcal{Q}_{tu}} {\esssup} \left\{E_{Q_t} \left[ - X \vert \mathcal{F}_t \right] - \alpha_{tu} (Q_t) \right\},
\end{equation}
where $Q_t$ corresponds to $(q(t,s))_{t \leq s \leq u}$ via~\eqref{eq: definition-Qt} and the penalty functional is given by
\begin{equation*} 
\alpha_{tu} (Q_t)= E_{Q_t} \left[\left. \int_t ^u g^*(t,s,q(t,s)) ds \right\vert \mathcal{F}_t \right],
\end{equation*}
with $g^*(t,s,\cdot)$ denoting the convex conjugate of $g(t,s,\cdot)$.
\end{theorem}

\noindent
\begin{proof}
The present proof extends a similar one of Barrieu and El Karoui~\cite{barrieu-el-karoui} to the case of BSVIEs.
Let $t, u \in [0,T]$ with $t \leq u$ be fixed arbitrarily. The evaluation of risk~\eqref{eq: rho_BSVIEg} can be written as
\begin{eqnarray*}
Y_t \hspace{-2mm}&=&\hspace{-2mm} -X + \hspace{-1mm}\int_t^u \hspace{-2mm} [g(t,s,Z(t,s))- q(t,s) \cdot Z(t,s)] ds - \hspace{-1mm}\int_t^u \hspace{-2mm} Z(t,s) (dB_s - q(t,s) ds) \\
&=&\hspace{-2mm} -X + \int_t^u [g(t,s,Z(t,s))- q(t,s) \cdot Z(t,s)] ds - \int_t^u Z(t,s) dB^{Q_t}_s
\end{eqnarray*}
and, by Girsanov Theorem, $B^{Q_t}_u \triangleq   B_u -B_t - \int_t^u q(t,s) ds$, $u \geq t$, is a Brownian motion with respect to $Q_t$ (see~\eqref{eq: definition-Qt}) with initial value $B^{Q_t}_t=0$.
Set now
\begin{equation*}
g^*(t,s,q) \triangleq \underset{ z \in \Bbb R^d }{\esssup} \{q \cdot z - g(t,s,z) \}, \, q \in \Bbb R^d,
\end{equation*}
for $s,t$ such that $s \geq t$. We obtain that, for any fixed $t$, $g^*(t,\cdot,q(t,\cdot))$ is adapted.
Furthermore,
\begin{eqnarray*}
Y_t &=& -X - \int_t^u [q(t,s) \cdot Z(t,s)-g(t,s,Z(t,s))] ds - \int_t^u Z(t,s) dB^{Q_t}_s \\
& \geq & -X - \int_t^u g^*(t,s,q(t,s)) ds - \int_t^u Z(t,s) dB^{Q_t}_s,
\end{eqnarray*}
where the last inequality follows by the definition of $g^*$. By taking the conditional expectation with respect to $Q_t$, it follows
$$
Y_t \geq  E_{Q_t} \left[ - X \vert \mathcal{F}_t \right] - E_{Q_t} \left[\left. \int_t^u g^*(t,s,q(t,s)) ds \right\vert \mathcal{F}_t \right],
$$
since $E_{Q_t} [\int_t^u Z(t,s) dB^{Q_t}_s \vert \mathcal{F}_t]=0$ because $(B^{Q_t}_s)_{s \geq t}$ is a Brownian motion with respect to $Q_t$. Consequently,
\begin{equation} \label{eq: dual-bsvie-diseq1}
Y_t \geq  \underset{Q_t \in \mathcal{Q}_{tu}}{\esssup} \Big\{ E_{Q_t} \big[  - X \vert \mathcal{F}_t \big] - E_{Q_t} \Big[  \int_t^u g^*(t,s,q(t,s)) ds \big\vert \mathcal{F}_t \Big] \Big\}.
\end{equation}
It remains to prove the converse inequality in~\eqref{eq: dual-bsvie-diseq1}. In a similar way as in Theorem 7.4 (i) and Lemma 7.5 of Barrieu and El Karoui~\cite{barrieu-el-karoui}, for $t$ fixed and any $Z(t,s)$, $s \geq t$, there exist some progressively measurable $q_Z(t,s)$ (associated to a $Q_{t,Z}$ via~\eqref{eq: definition-Qt}) such that $g(t,s,Z(t,s)) = q_Z(t,s) \cdot Z(t,s) - g^*(t,s,q_Z(t,s)) $. Hence,
\begin{eqnarray*}
Y_t &=& -X + \int_t^u g(t,s,Z(t,s)) ds - \int_t^u Z(t,s) dB_s  \\
&=& -X + \int_t^u [q_Z(t,s) \cdot Z(t,s) - g^*(t,s,q_Z(t,s))] ds - \int_t^u Z(t,s) dB_s \\
&=& -X - \int_t^u  g^*(t,s,q_Z(t,s)) ds - \int_t^u Z(t,s) dB^{Q_{t,Z}}_s,
\end{eqnarray*}
where $B^{Q_{t,Z}}_{s}= B_s-B_t-\int_t^s q_Z(t,r) dr$, $t\leq s \leq u$.
By taking the conditional expectation with respect to $Q_{t,Z}$, it holds that
\begin{eqnarray}
Y_t &=& E_{Q_{t,Z}} \left[ \left. -X - \int_t^u  g^*(t,s,q_Z(t,s)) ds \right\vert \mathcal{F}_t \right]  \notag\\
&\leq & \underset{Q_t \in \mathcal{Q}_{tu}}{\esssup} \left\{E_{Q_t} \left[ - X \vert \mathcal{F}_t \right] - E_{Q_t} \left[ \left. \int_t^u g^*(t,s,q(t,s)) ds \right\vert \mathcal{F}_t \right] \right\}. \label{eq: dual-bsvie-diseq2}
\end{eqnarray}
The thesis follows then by~\eqref{eq: dual-bsvie-diseq1} and~\eqref{eq: dual-bsvie-diseq2}.
\end{proof}
\smallskip

\subsubsection{Time-consistency}

The result below provides a necessary and sufficient condition on $g$ for sub time-consistency of $(\rho_{tu})_{t,u}$ induced by a BSVIE $(g)$.

\begin{theorem} \label{prop: sub-tc-bsvie}
Let $(\rho_{tu})_{t,u}$ be induced by a BSVIE with driver $g$ as in~\eqref{eq: rho_BSVIEg}.

\noindent a) The driver $g(t,v,z)$ is decreasing in $t$ for any $v \in [t,u],z \in \Bbb R^d$ (meaning that, for any $s \leq t$, $g(t,v,z) \leq g(s,v,z)$) if and only if $(\rho_{tu})_{t,u}$ satisfies sub time-consistency.

\noindent b) If the driver $g(t,\cdot,\cdot)$ is constant in $t$, then $(\rho_{tu})_{t,u}$ satisfies time-consistency.
\end{theorem}

\noindent
\begin{proof}
a) Assume that sub time-consistency holds.
By the dual representation of $(\rho_{tu})_{t,u}$ in Proposition~\ref{prop: dual-repres-bsvie}, the penalty term of $\rho_{tu}$ is given by
\begin{equation*}
\alpha_{tu} (Q_t)= E_{Q_t} \left[\left. \int_t ^u g^*(t,v,q(t,v)) dv \right\vert \mathcal{F}_t \right]
\end{equation*}
for any $Q_t \in \mathcal{Q}_{tu}$.
Let $s, t, u \in [0,T]$ with $s \leq t \leq u$ and let $Q_s \in \mathcal{Q}_{st}, Q_t \in \mathcal{Q}_{tu}$ be fixed arbitrarily. Set now $\bar{Q}=Q_s \cdot Q_t$ the pasting of $Q_s$ on $[s,t]$ and of $Q_t$ on $[t,u]$, hence $\bar{Q} \in \mathcal{Q}_{su}$. Denote by $q(s,v)$, $q(t,v)$ and $\bar{q}(s,v)$ the corresponding processes as in~\eqref{eq: definition-Qt}.
By applying the characterization of penalty term of sub time-consistent risk measures (see Proposition~\ref{prop: charact-sub time-cons}(iii)), we obtain
\begin{eqnarray*}
&&\hspace{-8mm} {\displaystyle  E_{\bar{Q}} \Big[ \int_s ^u g^*(s,v,\bar{q}(s,v)) dv \Big\vert \mathcal{F}_s \Big]} \\
&&\hspace{-8mm} {\displaystyle\leq  E_{\bar{Q}} \Big[ \int_s ^t g^*(s,v,\bar{q}(s,v)) dv \Big\vert \mathcal{F}_s \Big] + E_{\bar{Q}} \Big[  E_{\bar{Q}} \Big[ \int_t ^u g^*(t,v,\bar{q}(t,v)) dv \Big\vert \mathcal{F}_t \Big] \Big\vert \mathcal{F}_s \Big]} ,
\end{eqnarray*}
hence
\begin{eqnarray}
&&\hspace{-8mm} {\displaystyle E_{\bar{Q}} \Big[ \int_t ^u g^*(s,v,\bar{q}(s,v)) dv \Big\vert \mathcal{F}_s \Big]  \leq  E_{\bar{Q}} \Big[  E_{\bar{Q}} \Big[ \int_t ^u g^*(t,v,\bar{q}(t,v)) dv \Big\vert \mathcal{F}_t \Big] \Big\vert \mathcal{F}_s \Big]} \notag \\
&&\hspace{-8mm} {\displaystyle E_{Q_s} \hspace{0mm}\Big[ \hspace{0mm}  E_{Q_t} \hspace{0mm} \Big[\hspace{0mm}  \int_t ^u \hspace{-2mm} g^*(s,v,\bar{q}(s,v)) dv \Big\vert \hspace{0mm} \mathcal{F}_t \Big] \Big\vert \hspace{0mm} \mathcal{F}_s \Big]  \hspace{0mm} \leq E_{Q_s} \hspace{0mm} \Big[ \hspace{0mm} E_{Q_t} \hspace{0mm} \Big[\hspace{0mm}  \int_t ^u \hspace{-2mm} g^*(t,v,\bar{q}(t,v)) dv \Big\vert \hspace{0mm} \mathcal{F}_t \Big] \Big\vert \hspace{0mm} \mathcal{F}_s \Big]} \notag\\
&&\hspace{-8mm}  {\displaystyle E_{Q_s} \Big[  E_{Q_t} \Big[ \int_t ^u \Big[ g^*(t,v,\bar{q}(t,v)) - g^*(s,v,\bar{q}(s,v))\Big] dv \Big\vert \mathcal{F}_t \Big] \Big\vert \mathcal{F}_s \Big]  \geq 0.} \label{eq: condition g-bsvie}
\end{eqnarray}
Since~\eqref{eq: condition g-bsvie} should hold for any $s \leq t \leq u$ and any $Q_s \in \mathcal{Q}_{st}, Q_t \in \mathcal{Q}_{tu}$, it follows that
\begin{equation*}
g^*(t,v,\bar{q}(t,v)) \geq  g^*(s,v,\bar{q}(s,v)) \quad \mbox{ for any } s \leq t, v \in [t,u] , \, \bar{q} \in \Bbb{R}^d.
\end{equation*}
Hence, $g(t, \cdot, \cdot) \leq g(s, \cdot, \cdot)$ for any $s \leq t$.\smallskip

Conversely, assume that $g(t,\cdot ,\cdot)$ is decreasing in $t$, in the meaning specified above. We are going to prove that
\begin{equation} \label{eq: sub tc-BSVIE-proof}
\rho_{st} (-\rho_{tu} (X))= \mathcal{E}^{g,V} (\mathcal{E}^{g,V} (-X \vert \mathcal{F}_t) \vert \mathcal{F}_s) \leq \mathcal{E}^{g,V} (-X \vert \mathcal{F}_s)=\rho_{su} (X)
\end{equation}
for any $s \leq t \leq u$ and $X \in L^2(\mathcal{F}_u)$.
By definition of $\mathcal{E}^{g,V}$, the left-hand and right-hand sides of the previous equation can be rewritten as follows:
\begin{eqnarray}
&& \mathcal{E}^{g,V} (\mathcal{E}^{g,V} (-X \vert \mathcal{F}_t) \vert \mathcal{F}_s)  \notag \\
=&&  \mathcal{E}^{g,V} (-X \vert \mathcal{F}_t)+ \int_s ^t g(s,v,Z(s,v)) \, dv - \int_s ^t Z(s,v) \, dB_v \notag \\
=&& -X + \int_t ^u g(t,v,\tilde{Z} (t,v) \, dv - \int_t^u  \tilde{Z}(t,v) \, dB_v +  \int_s ^t  g(s,v,Z(s,v)) \, dv -  \int_s ^t  Z(s,v) \, dB_v  \notag \\
=&&  -X +\int_s ^u \Big[g(s,v,Z(s,v)) 1_{[s,t]}(v) + g(t,v,\tilde{Z} (t,v)) 1_{(t,u]}(v) \Big] \, dv \notag \\
&&  - \int_s^u \big[ Z(s,v) 1_{[s,t]} (v) +\tilde{Z}(t,v)1_{(t,u]} (v) \big] \, dB_v \label{eq: sub tc-BSVIE-proof-1}
\end{eqnarray}
and
\begin{equation}  \label{eq: sub tc-BSVIE-proof-2}
\mathcal{E}^{g,V} (-X \vert \mathcal{F}_s)  = -X +\int_s ^u g(s,v,\hat{Z}(s,v) \, dv - \int_s^u \hat{Z} (s,v) \, dB_v.
\end{equation}
Furthermore, by decreasing monotonicity of $g(t,\cdot,\cdot)$ in $t$ and by~\eqref{eq: sub tc-BSVIE-proof-1} it follows that
\begin{eqnarray*}
&\;&\mathcal{E}^{g,V} (\mathcal{E}^{g,V} (-X \vert \mathcal{F}_t) \vert \mathcal{F}_s) \\
&\leq & -X +\int_s ^u \left [g(s,v,Z(s,v)) 1_{[s,t]}(v) + g(s,v,\tilde{Z} (t,v)) 1_{(t,u]}(v) \right] \, dv \\
&\:&- \int_s^u \left[ Z(s,v) 1_{[s,t]} (v) +\tilde{Z}(t,v)1_{(t,u]} (v) \right] \, dB_v .
\end{eqnarray*}
Set now
\begin{equation*}
\bar{Z}_t (s,v)=  Z(s,v) 1_{[s,t]}(v) +\tilde{Z}(t,v)1_{(t,u]}(v), \quad v \geq s .
\end{equation*}
By the arguments above,
\begin{equation} \label{eq: sub tc-BSVIE-proof-3}
\mathcal{E}^{g,V} (\mathcal{E}^{g,V} (-X \vert \mathcal{F}_t) \vert \mathcal{F}_s)  \leq -X +\int_s ^u g(s,v,\bar{Z}_t(s,v) \, dv - \int_s^u  \bar{Z}_t(s,v)  \, dB_v.
\end{equation}
Sub time-consistency~\eqref{eq: sub tc-BSVIE-proof} then follows by comparing~\eqref{eq: sub tc-BSVIE-proof-2} and~\eqref{eq: sub tc-BSVIE-proof-3} and by the uniqueness of the solution of a BSVIE.

\smallskip
\noindent
b) The case of time-consistency can be obtained by replacing inequalities with equalities in the proof above.
\end{proof}

\medskip
Notice that a BSVIE of the form~\eqref{eq: sub tc-BSVIE-proof-2}, that is with a final condition $X$ that does not depend on $t$ and with a driver $g$ independent of $y$, reduces to a BSDE when the driver $g(t,v,z)$ is constant in $t$.

\bigskip
An alternative proof of item a) could be given in terms of a Converse Comparison Theorem for BSVIEs (similarly to Theorem 4.1 of Briand et al.\cite{BCHMPeng} for BSDEs) if that result was available.
However, while for BSDEs the Converse Comparison Theorem has been proved by Briand et al.~\cite{BCHMPeng}, for BSVIEs we are not aware of any similar result. We prove below that a Converse Comparison Theorem holds also for BSVIEs.\smallskip

We recall that a BSVIE~\eqref{eq: BSVIE} is related to the following family of BSDEs parameterized by $t$:
\begin{equation} \label{eq: def eta}
\eta (r;t, \xi_t)= \xi _t + \int_r^T \bar{g}(v,\zeta(v;t);t) dv - \int_r^T \zeta(v;t) dB_v, \quad r \in [t,T]
\end{equation}
where
\begin{equation*}
\zeta(v;t)=Z(t,v); \quad \bar{g}(v,\zeta(v;t);t)=g(t,v,Z(t,v))\quad \mbox{and} \quad Y_t= \eta(t;t, \xi_t).
\end{equation*}
We remark again that the drivers above do not depend on $y$. In this way, translation invariance of the associated risk measure is not guaranteed.
We are then able to prove a Converse Comparison Theorem for BSVIEs.

\begin{theorem}(Converse Comparison Theorem for BSVIEs) \label{prop: weak-converse-comparison-BSVIE}
Let $Y_t^1$ and $Y_t^2$ be two BSVIEs as in~\eqref{eq: BSVIE} with drivers $g^1$ and $g^2$ and terminal condition $\xi_t$ and let $\eta^1$ and $\eta^2$ be the corresponding families defined in~\eqref{eq: def eta}.

\noindent
a) If $\eta^1 (r;t, \xi_t) \leq \eta^2 (r;t, \xi_t)$ for any $t \in [0,T]$, $r \in [t,T]$ and $\xi_t \in L^2 (\mathcal{F}_T)$, then $g^1 \leq g^2$, that is $g^1(t,v,z) \leq g^2(t,v,z)$ for any $t \in [0,T], v \in [t,T], z \in \Bbb R^d$.

\noindent
b) If $Y_t^{1,X} \leq Y_t^{2,X}$ for any $t \in [0,T]$ and $X \in L^2 (\mathcal{F}_T)$ and if $g^i(t,v,0)=0$ for $i=1,2$, then $g^1 \leq g^2$.
\end{theorem}

We recall that, for BSDEs, it is sufficient to require that $\mathcal{E}^{g^1}(X) \leq \mathcal{E}^{g^2}(X)$ holds for any $X \in L^2 (\mathcal{F}_T)$ to guarantee that $\mathcal{E}^{g^1}_t(X) \leq \mathcal{E}^{g^2}_t(X)$ for any $X \in L^2 (\mathcal{F}_T)$. This is mainly due to time-consistency of $g$-expectations. See Theorem 4.4 of Briand et al.~\cite{BCHMPeng} for details. The proof of item b) is based on the argument above.\medskip

\noindent
\begin{proof}
a) By~\eqref{eq: def eta}, $\eta^1(\cdot;t)$ and $\eta^2(\cdot; t)$ are two BSDEs parameterized by $t$.
By applying the Converse Comparison Theorem for BSDEs (see Theorem 4.1 of Briand et al.~\cite{BCHMPeng} and Lemma 2.1 of Jiang~\cite{jiang}) to $\eta^1(\cdot;t), \eta^2(\cdot; t)$, it follows that $\bar{g}^1(v,\zeta;t) \leq \bar{g}^2 (v, \zeta;t)$ for any $v \geq r \geq t$ and $\zeta \in \Bbb R^d$. Hence $g^1(t,v,z) \leq g^2(t,v,z)$ for any $t \in [0,T], v \in [t,T], z \in \Bbb R^d$.

\smallskip
\noindent
b) First of all, we prove that, given two BSDEs with drivers $g^1$ and $g^2$ satisfying $g^i(v,0)=0$ and a fixed $t$,
\begin{eqnarray}
&&\mathcal{E}^{g^1}(X | \mathcal{F}_t) \leq \mathcal{E}^{g^2}(X | \mathcal{F}_t) \mbox{ for any } X \in L^2 (\mathcal{F}_T) \notag \\
&&\Rightarrow \mathcal{E}^{g^1}(X | \mathcal{F}_r) \leq \mathcal{E}^{g^2}(X | \mathcal{F}_r) \mbox{ for any } X \in L^2 (\mathcal{F}_T) \mbox{ and } r \in [t,T]. \label{eq: star}
\end{eqnarray}
This implication extends a similar one in the proof of Theorem 4.4 of Briand et al.~\cite{BCHMPeng} where $t=0$.
Assume now that, for a given fixed $t$,  $\mathcal{E}^{g^1}(X | \mathcal{F}_t) \leq \mathcal{E}^{g^2}(X | \mathcal{F}_t)$ for any $X \in L^2 (\mathcal{F}_T)$.
By using a similar approach to those in Theorem 4.4 of Briand et al.~\cite{BCHMPeng} and Lemma 4.5 of Coquet et al.~\cite{CHMPeng},
\begin{eqnarray*}
0=\mathcal{E}^{g^1}\left(\left. X - \mathcal{E}^{g^1}(X | \mathcal{F}_r) \right| \mathcal{F}_t\right) &\leq & \mathcal{E}^{g^2}\left(\left. X - \mathcal{E}^{g^1}(X | \mathcal{F}_r) \right| \mathcal{F}_t\right) \\
&=&\mathcal{E}^{g^2}\left( \left. \mathcal{E}^{g^2}\left( \left. X - \mathcal{E}^{g^1}(X | \mathcal{F}_r) \right| \mathcal{F}_r\right) \right | \mathcal{F}_t\right) \\
&=&\mathcal{E}^{g^2}\left( \left. \mathcal{E}^{g^2} (X | \mathcal{F}_r) - \mathcal{E}^{g^1}(X | \mathcal{F}_r) \right| \mathcal{F}_t\right)
\end{eqnarray*}
for any $X \in L^2 (\mathcal{F}_T)$ and $r \in [t,T]$, where the first and the last equality are due to translation invariance.
Fix now $r \in [t,T]$ arbitrarily and consider $\xi=X1_A$ with $A=\{\mathcal{E}^{g^2} (X | \mathcal{F}_r) < \mathcal{E}^{g^1}(X | \mathcal{F}_r) \} \in \mathcal{F}_r$. On the one hand, by the arguments above,
\begin{equation*}
\mathcal{E}^{g^2}\left( \left. \mathcal{E}^{g^2} (\xi | \mathcal{F}_r) - \mathcal{E}^{g^1}(\xi | \mathcal{F}_r) \right| \mathcal{F}_t\right) \geq 0.
\end{equation*}
On the other hand,
\begin{eqnarray*}
\mathcal{E}^{g^2}\left( \left. \mathcal{E}^{g^2} (\xi | \mathcal{F}_r) - \mathcal{E}^{g^1}(\xi | \mathcal{F}_r) \right| \mathcal{F}_t\right) \hspace{-1mm} &=& \hspace{-1mm} \mathcal{E}^{g^2}\left(\left. 1_A  \left( \mathcal{E}^{g^2} (X | \mathcal{F}_r) - \mathcal{E}^{g^1}(X | \mathcal{F}_r) \right) \right| \mathcal{F}_t\right) \\
\hspace{-1mm} & \leq & \hspace{-1mm} \mathcal{E}^{g^2}(0| \mathcal{F}_t) =0,
\end{eqnarray*}
because $\mathcal{E}^{g^i}(X 1_A| \mathcal{F}_r) = 1_A \mathcal{E}^{g^i}(X | \mathcal{F}_r) $ for any $A \in \mathcal{F}_r$ holds because of normalization of $g^i$ (see Peng~\cite{peng97}) and the last equality is due to $g^2(t,v,0)=0$.
Hence,
\begin{equation*}
\mathcal{E}^{g^2}\left(\left. 1_A  \left( \mathcal{E}^{g^2} (X | \mathcal{F}_r) - \mathcal{E}^{g^1}(X | \mathcal{F}_r) \right) \right| \mathcal{F}_t\right)=0.
\end{equation*}
Since $1_A  \left( \mathcal{E}^{g^2} (X | \mathcal{F}_r) - \mathcal{E}^{g^1}(X | \mathcal{F}_r) \right) \leq 0$ and strictly negative with probability equal to $P(A)$, strict monotonicity of conditional $g$-expectation implies that
\begin{equation*}
1_A  \left( \mathcal{E}^{g^2} (X | \mathcal{F}_r) - \mathcal{E}^{g^1}(X | \mathcal{F}_r) \right)=0, \quad P-a.s.,
\end{equation*}
then $\mathcal{E}^{g^2} (X | \mathcal{F}_r) \leq \mathcal{E}^{g^1}(X | \mathcal{F}_r)$, $P$-a.s., for $r \in [t,T]$. This concludes the proof of~\eqref{eq: star}.
\smallskip

Going back to BSVIEs, denote by $\eta^{t,X}_{i,r}=\eta^i (r;t, X)$ and by $\bar{g}^t_i= \bar{g}^i(\cdot,\cdot;t)$ for $i=1,2$, where $\eta$ and $\bar{g}$ are defined above.
Fix now $t$ arbitrarily. Assuming that $Y_t^{1,X} \leq Y_t^{2,X}$ for any $X \in L^2 (\mathcal{F}_T)$ is equivalent to assuming that $\eta^{t,X}_{1,t} \leq \eta^{t,X}_{2,t}$ or, also, to $\mathcal{E}^{\bar{g}^t_1} (X | \mathcal{F}_t)\leq \mathcal{E}^{\bar{g}^t_2} (X | \mathcal{F}_t)$ for any $X \in L^2 (\mathcal{F}_T)$. By~\eqref{eq: star}, it follows that $Y_t^{1,X} \leq Y_t^{2,X}$ for any $X \in L^2 (\mathcal{F}_T)$ implies that $\mathcal{E}^{\bar{g}^t_1} (X | \mathcal{F}_r)\leq \mathcal{E}^{\bar{g}^t_2} (X | \mathcal{F}_r)$ for any $r \in [t,T]$, hence $\eta^{t,X}_{1,r} \leq \eta^{t,X}_{2,r}$ for any $r \in [t,T]$. The thesis then follows by item a) where it is not necessary to have $\xi_t$ but it is enough to consider $X \in L^2 (\mathcal{F}_T)$.
\end{proof}

\subsubsection{H-longevity and BSVIEs}

The next results investigate under which conditions on the driver $g$ h-longevity of $(\rho_{tu})_{t,u}$ is fulfilled.

\begin{corollary}
If $g(t,\cdot,\cdot)$ is decreasing in $t$ and $g(t,v,0) \geq 0$ for any $t \leq v$, then h-longevity holds.
\end{corollary}

\noindent
\begin{proof}
Since $g(t,\cdot ,\cdot)$ is decreasing in $t$, sub time-consistency follows by Theorem~\ref{prop: sub-tc-bsvie}. We then show here $\rho_{tu}(0) \geq 0$ for any $t \leq u$, so that we can conclude by Remark~\ref{rem: tc and longevity}.
We prove now that $g(t,v,0) \geq 0$ for any $t \leq v$ implies that $\rho_{tu}(0) \geq 0$ for any $t \leq u$.
Indeed,
\begin{equation*}
\rho_{tu}(0)=Y_t= \int_t^u g(t,v,Z(t,v)) dv - \int_t ^u Z(t,v) dv.
\end{equation*}
Since the solution is $(Y_t=\int_t ^u g(t,v,0) dv; Z(t,v)=0)$, then $\rho_{tu}(0) \geq 0$ by the assumption $g(t,v,0) \geq 0$ for any $t \leq v$.
\end{proof}

\medskip
\noindent
In reality, we can do something more.
\begin{proposition} \label{prop: longevity-BSVIE}
If $g(s,v,0) \geq 0$ for any $s \leq v$, then h-longevity holds.
Furthermore, $\gamma(s,t,u,X)=E_{\widetilde{Q}_{s,X}} \left[\int_t^u g(s,v,0) dv | \mathcal{F}_s\right]$ for any $s \leq t \leq u$, where $\widetilde{Q}_{s,X}$ is a suitable probability measure depending on $X$, with density
\begin{equation*} 
\frac{d \widetilde{Q}_{s,X}}{dP}\triangleq \exp \left\{ - \frac 12 \int_s^u \vert \Delta_z g(s,v)\vert ^2 dv + \int_s^u \Delta_z g(s,v) dB_v \right\}.
\end{equation*}
Here above $\Delta_z g(v)= (\Delta_z ^i g(v))_{i=1,...,d}$ and
\begin{equation*}
\Delta_z ^i g(v) \triangleq \frac{g(s,v,Z^u(s,v))- g(s,v,\bar{Z}^t(s,v))}{d \left( Z^{u,i} (s,v)- \bar{Z}^{t,i} (s,v) \right)} 1_{\{Z^{u,i} (s,v)\neq \bar{Z}^{t,i} (s,v)\}}.
\end{equation*}
\end{proposition}

\noindent
\begin{proof}
Let $s \leq t \leq u$ and let $X \in L^2 (\mathcal{F}_t)$ be fixed arbitrarily.
The risk measures $\rho_{st}$ and $\rho_{su}$ satisfy, respectively, the following BSVIEs:
\begin{eqnarray*}
\rho_{st}(X)&=& -X+ \int_s^t g(s,v,Z^t(s,v)) dv - \int_s^t Z ^t (s,v) dB_v \\
\rho_{su}(X)&=& -X+ \int_s^u g(s,v,Z^u (s,v)) dv - \int_s^u Z ^u (s,v) dB_v.
\end{eqnarray*}
Set now
\begin{equation*}
\bar{Z}^t (s,v)=\left\{
\begin{array}{rl}
Z^t (s,v);& s \leq v \leq t \\
0;& t<v \leq u
\end{array}
\right. ;
\quad \widetilde{Z} (s,v)= Z^u (s,v) - \bar{Z}^t (s,v).
\end{equation*}
Then
\begin{eqnarray}
\hspace{-2mm} &&\rho_{su}(X)-\rho_{st}(X) \notag\\
\hspace{-2mm} &=&  \int_s^u [g(s,v,Z^u (s,v))- g(s,v,\bar{Z}^t(s,v))] dv + \int_t^u g(s,v,\bar{Z}^t (s,v)) dv \notag \\
\hspace{-2mm} && - \int_s^u [Z ^u (s,v)- \bar{Z}^t (s,v)] dB_v - \int_t^u \bar{Z} ^t (s,v) dB_v \notag \\
\hspace{-2mm} &=&  \int_s^u [g(s,v,Z^u (s,v))- g(s,v,\bar{Z}^t(s,v))] dv + \int_t^u g(s,v,0) dv \notag \\
\hspace{-2mm} && - \int_s^u \widetilde{Z} (s,v) dB_v  \notag \\
\hspace{-2mm} &=&  \int_s^u \Delta_z g(s,v) \cdot \widetilde{Z} (s,v) dv - \int_s^u \widetilde{Z} (s,v) dB_v+ \int_t^u g(s,v,0) dv. \label{eq: bsvie-longevity-1}
\end{eqnarray}
Furthermore,~\eqref{eq: bsvie-longevity-1} can be rewritten as a linear BSVIE
\begin{equation} \label{eq: linear BSVIE-longevity}
\delta \rho_s= \Gamma^{t ,u}_s + \int_s^u \Delta_z g(s,v) \cdot \widetilde{Z} (s,v) dv - \int_s^u \widetilde{Z} (s,v) dB_v,
\end{equation}
where $\delta \rho_s \triangleq \rho_{su}(X)-\rho_{st}(X)$ and $\Gamma^{t,u}_s \triangleq \int_t^u g(s,v,0) dv$ represents the final condition at time $u$ (which depends on $t$ and also on $s$).

Since $\Gamma^{t,u}\geq 0$ for any $t$ by hypothesis and $\Delta_z g(s,v) \in \Bbb H^2_{[0,T]} (\Bbb R^d)$ by the assumption of $g$ Lipschitz in $z$, by the Comparison Theorem on BSVIEs (see Corollary 3.3 of Wang and Yong~\cite{wang-yong}) it follows the longevity, i.e. $\delta \rho_s \geq 0$ for any $s \leq t$.
Moreover, by applying Girsanov Theorem,~\eqref{eq: linear BSVIE-longevity} becomes
\begin{eqnarray*}
\delta \rho_s &=&  \int_t^u g(s,v,0) dv + \int_s^u \Delta_z g(s,v) \cdot \widetilde{Z}(s,v) dv - \int_s^u \widetilde{Z} (s,v) dB_v \\
&=& \int_t^u g(s,v,0) dv - \int_s^u \widetilde{Z}(s,v) dB^{\widetilde{Q}_{s,X}}_v,
\end{eqnarray*}
where $B^{\widetilde{Q}_{s,X}}_v \triangleq B_v-B_s- \int_s^v \Delta_z g(s,v) dv$, $v \geq s$, is a Brownian motion with respect to $\widetilde{Q}_{s,X}$ with initial value $B^{\widetilde{Q}_{s,X}}_s=0$. Hence, by taking the conditional expectation with respect to $\widetilde{Q}_{s,X}$,
\begin{equation*}
\gamma(s,t,u,X)=\delta \rho_s=E_{\widetilde{Q}_{s,X}} \left[\left. \int_t^u g(s,v,0) dv \right| \mathcal{F}_s\right].
\end{equation*}
It then follows that $\rho_{su}(X)-\rho_{st}(X)=E_{\widetilde{Q}_{s,X}} \left[\int_t^u g(s,v,0) dv | \mathcal{F}_s\right]$.
\end{proof}
\medskip

As discussed in Sections~\ref{sec: longevity} and~\ref{sec: BSDEg}, $\gamma$ may depend on the length of the time interval $[u,t]$, that is, $\gamma (s,t,u,X)=\gamma_{s,t}(h,X)$ with $h=u-t$ or, even, $\gamma_{s,t}(h)$ independent from $X$. The following example provides some cases covering the situation above.

\begin{example} \label{ex: longevity-bsvie}
Let $g(s,v,0) \geq 0$ for any $v \in [s,T]$. Hence, by the result above, h-longevity holds and $\gamma(s,t,u,X)=E_{\widetilde{Q}_{s,X}} \left[\int_t^u g(s,v,0) dv | \mathcal{F}_s\right]$ for any $X \in L^2(\mathcal{F}_t)$.\smallskip

a) If $g(s,v,0)=c_s$ for any $v \in [s,T]$, with $c_s\geq 0$, then $c_s$ is necessarily $\mathcal{F}_s$-measurable (since it should be measurable for any $v \geq s$) and, consequently,
\begin{equation*}
\gamma(s,t,u,X)=E_{\widetilde{Q}_{s,X}} \left[\left. \int_t^u g(s,v,0) dv \right| \mathcal{F}_s\right]= (u-t)c_s.
\end{equation*}
In other words, $\gamma$ only depends on the evaluation time $s$ and on $h=u-t$, that is, roughly speaking, on the length of the time interval over which there is an uncorrect use of the risk measure ($\rho_{su}$ versus $\rho_{st}$). \smallskip

b) If $g(s,v,0)=\exp(-r_s \, v)$ for any $v \in [0,T]$, with $r_s\geq 0$, then $r_s$ is necessarily $\mathcal{F}_s$-measurable (for the same arguments as above) and, consequently,
\begin{equation*}
\gamma(s,t,u,X)=\frac{e^{-r_s t} \left( 1-e^{-r_s (u-t)} \right)}{r_s}.
\end{equation*}
Hence $\gamma$ depends on the evaluation time $s$, on the ``right'' time horizon $t$ (referring to the measurability of $X$) and on the length of the time interval $[t,u]$.
Compared to the BSDE case (see Example~\ref{ex: longevity-bsde}), here $\gamma$ depends also on the evaluation time $s$. This is not surprising for BSVIEs.
\end{example}

Finally, we provide two examples of BSVIEs: the former with a linear driver, the latter going beyond the Lipschitz case and similarly to Section~\ref{sec: entrpic BSDE}.
\begin{example}\label{ex: linear Volterra}
Consider the driver $g(t,s,z)=a(t,s) \cdot z + b(t,s)$ for any $0 \leq t \leq s \leq T$ and $z \in \Bbb R^d$, where the $\Bbb R^d$-valued process $a(t,s)$ and $1$-dimensional process $b(t,s)$ are given.
By applying Girsanov Theorem, in the same line of Hu and {\O}ksendal~\cite{hu-oksendal}, the BSVIE associated to the linear driver above becomes
\begin{eqnarray*}
Y_t&=&-X+ \int_t ^T [a(t,s) \cdot Z(t,s) + b(t,s)] ds - \int_t^T Z(t,s) dB_s \\
&=&-X+ \int_t ^T b(t,s) ds - \int_t^T Z(t,s) dB^{\widetilde{Q}_t}_s \\
\end{eqnarray*}
where $\frac{d \widetilde{Q}_{t}}{dP}\triangleq \exp \left\{ - \frac 12 \int_t^T \vert a(t,s)\vert^2 ds + \int_t^T a(t,s) dB_s \right\}$ and $B^{\widetilde{Q}_{t}}_u \triangleq B_u- B_t- \int_t^u  a(t,s) ds$, for $u \geq t$, is a Brownian motion with respect to $\widetilde{Q}_{t}$ with initial value $B^{\widetilde{Q}_{t}}_t=0$. Hence, by taking the conditional expectation with respect to $\widetilde{Q}_{t}$, it holds that
\begin{equation*}
Y_t=  E_{\widetilde{Q}_{t}} \left[ \left. -X+ \int_t^T b(t,s) ds \right| \mathcal{F}_t\right].
\end{equation*}
Choosing $b(t,s) \geq 0$, the h-longevity holds.
\end{example}

\begin{example}\label{ex: quadratic Volterra}
Consider the driver $g(t,s,z)=b(t)\frac{\vert z \vert ^2}{2} +a(t,s)$ for any $0 \leq t \leq s \leq T$ and $z \in \Bbb R^d$, where the deterministic function $b$ is positive and the process $a$ is given.
Hence
\begin{eqnarray*}
Y_t&=&-X+ \int_t ^T \left[ b(t) \frac{\vert Z(t,s)\vert ^2}{2} +a(t,s)\right] ds - \int_t^T Z(t,s) dB_s \\
&=&-X+  \int_t ^T a(t,s) ds+ \int_t ^T b(t) \frac{\vert Z(t,s) \vert ^2}{2}  ds - \int_t^T Z(t,s) dB_s
\end{eqnarray*}
and, following the same arguments of~\cite{yong-recursive}, Example 3.1, it follows that
\begin{equation*}
Y_t= \frac{1}{b(t)} \ln E_P \left[\exp\left\{-b(t) \left( X-\int_t ^T a(t,s) ds \right) \right\} \Big\vert \mathcal{F}_t \right].
\end{equation*}
Whenever $a(t,s)$ is deterministic, $Y_t$ becomes
\begin{equation*}
Y_t= \frac{1}{b(t)} \ln E_P \left[e^{-b(t) X}   \Big\vert \mathcal{F}_t   \right] + \int_t^T a(t,s) ds,
\end{equation*}
that is a translation of the usual entropic risk measure. Choosing $a(t,s)>0$, {\color{blue}(strict)} h-longevity holds.
\end{example}

\subsection{Risk measures generated by a family of BSVIEs}

Suppose that, for any $t \leq u$, the risk measure $\rho_{tu}$ comes from a BSVIE with a driver $g_u$ depending on the maturity $u$. This means that
\begin{equation} \label{eq: rho_BSVIEg-family}
\rho_{tu}^{\mathcal{G}} (X)= \mathcal{E}^{g_u,V} \left(-X \vert \mathcal{F}_t \right), \quad X \in L^{2} (\mathcal{F}_u),
\end{equation}
where $\mathcal{E}^{g_u,V} \left(\xi \vert \mathcal{F}_t \right)$ denotes the $Y$-component of the solution $(Y_t,Z(t,s))_{t,s \in [0,T], s\geq t}$ of the following BSVIE with driver $g_u$:
\begin{equation} \label{eq: BSVIEg-family}
Y_t= \xi + \int_t^u g_u(t,s,Z(t,s)) ds - \int_t^u Z(t,s) dB_s.
\end{equation}
Assume now that $\mathcal{G}=(g_u)_{u \in [0,T]}$ is a family of drivers depending on the maturity $u$, independent of $y$, Lipschitz, and convex in $z$. Each risk measure $\rho^{\mathcal{G}}_{tu}$ is of type~\eqref{eq: rho_BSVIEg}.
\bigskip

By applying Theorem~\ref{prop: dual-repres-bsvie} with a driver $g_u$ parameterized by $u$, in a Brownian setting $\rho_{tu}^{\mathcal{G}}$ can be represented as
\begin{equation*} 
\rho_{tu}(X)=\rho_{tu}^{\mathcal{G}}(X)=\underset{Q_t \in \mathcal{Q}_{tu}}{\esssup} \left\{ E_{Q_t} \left[-X \vert \mathcal{F}_t \right] -\alpha^{\mathcal{G}}_{tu} (Q_t) \right\}, \quad X \in L^{2} (\mathcal{F}_u),
\end{equation*}
where $\mathcal{Q}_{tu}$ is defined in~\eqref{eq: definition-Qt},
$g_u^*(t,s,\cdot)$ denotes the convex conjugate of $g_u(t,s,\cdot)$ and the minimal penalty functional is given by
\begin{equation} \label{eq: dual-repr-bsvie-family}
\alpha^{\mathcal{G}}_{tu} (Q_t)= E_{Q_t} \left[\left. \int_t ^u g_u^*(t,s,q(t,s)) ds \right\vert \mathcal{F}_t \right].
\end{equation}
Furthermore, if $g_u(t,s, 0) = 0$ for any $t\leq s \leq u$ then $\rho_{tu}(0) = 0$ for any $u$. Differently from the risk measures generated by a single BSVIE but similarly to those generated by a family of BSDEs, in general $g_u(t,s, 0) = 0$ for any $t\leq s \leq u$ does \textit{not} imply the restriction property.

The following result shows that, for $g_u(t,s,0)=0$ for any $t,s, u$, the restriction property is satisfied only for risk measures induced by a single BSVIE. This result is not surprising in view of Proposition~\ref{prop: restriction-BSDE-family} for the case of BSDEs with a family of drivers.

\begin{proposition} \label{prop: restriction-BSVIE-family}
Let $g_u(t,s,0)=0$ for any $t,s,u$.
The restriction property~\eqref{restriction} holds if and only if $g_u$ is constant in $u$.
\end{proposition}

\noindent
\begin{proof}
Assume that the restriction property holds, i.e. $\rho_{tu}(X)=\rho_{tv}(X)$ for any $t \leq u \leq v$ and $X \in L^2(\mathcal{F}_u)$.
Similarly to~\eqref{eq: def eta}, denote by
\begin{equation*}
\eta^{t,X}_{ru}= -X + \int_r^u \bar{g}_u^t(s,\zeta_v^t) dv - \int_r^u \zeta_v^t dB_v, \quad r \in [t;u]
\end{equation*}
where
\begin{equation*}
\zeta^t_v=Z(t,v); \quad \bar{g}_u^t(v,\zeta^t_v)=g(t,v,Z(t,v))\quad \mbox{and} \quad \rho_{tu}(X)= \eta^{t,X}_{tu}.
\end{equation*}
Assumptions on $g_u$ guarantee that $\bar{g}^t_u(s,0)=0$ for any $s$ and that $\bar{g}^t_u(s,\cdot)$ is continuous in $s$.
Proceeding as in the proof of the Converse Comparison Theorem of Briand et al.~\cite{BCHMPeng}, Thm. 4.1, and of Jiang~\cite{jiang}, Lemma 2.1,
\begin{eqnarray*}
\bar{g}^t_u(s,z)=\lim _{\varepsilon \to 0} \frac{\eta^{t,z \cdot (B_{s + \varepsilon} -B_{\varepsilon})}_{su}}{\varepsilon}\\
\bar{g}^t_v(s,z)=\lim _{\varepsilon \to 0} \frac{\eta^{t,z \cdot (B_{s + \varepsilon} -B_{\varepsilon})}_{sv}}{\varepsilon},
\end{eqnarray*}
with convergence in $L^p$ with $p \in [1,2)$, for any $z \in \Bbb R^d$, $u \leq v$ and $s \in [0,u]$. By extracting a subsequence to obtain convergence $P$-a.s. and passing to the limit as $\varepsilon \to 0$, it holds that
\begin{equation*}
\frac{\eta^{t,z \cdot (B_{s + \varepsilon} -B_{\varepsilon})}_{sv}}{\varepsilon}=\frac{\eta^{t,z \cdot (B_{s + \varepsilon} -B_{\varepsilon})}_{su}}{\varepsilon} \longrightarrow \bar{g}^t_u(s,z), \quad \epsilon \to 0, \:P\mbox{-a.s.}
\end{equation*}
where the equality is due to restriction. The thesis then follows because
\begin{equation*}
\frac{\eta^{t,z \cdot (B_{s + \varepsilon} -B_{\varepsilon})}_{sv}}{\varepsilon} \longrightarrow \bar{g}^t_v(s,z), \quad \epsilon \to 0, \: P\mbox{-a.s.}
\end{equation*}
The converse follows immediately by Proposition~\ref{prop: norm-restrict-BSVIE-g}.
\end{proof}

\subsubsection{Time-consistency}

The following result provides a necessary and sufficient condition for a fully-dynamic risk measure induced by a family of BSVIEs to satisfy sub time-consistency. Note that the condition on the monotonicity of $g_{\cdot}(t,\cdot,\cdot)$ is the same as for a BSVIE with a single driver, while the condition on the monotonicity of the family $g_u$ is new.

\begin{proposition} \label{prop: sub-tc-bsvie-family}
Let $(\rho_{tu})_{t,u}$ be induced by a BSVIE with a family of drivers $(g_u)_{u \in [0,T]}$ as in~\eqref{eq: rho_BSVIEg-family}.

\noindent
a) $(\rho_{tu})_{t,u}$ satisfies sub time-consistency if and only if both the family $\mathcal{G}$ is increasing and
 $g_{\cdot}(t,\cdot,\cdot)$ is decreasing in $t$.

\noindent
b) $(\rho_{tu})_{t,u}$ satisfies time-consistency if and only if $\mathcal{G} \hspace{-1mm}= \hspace{-1mm} \{g \}$ and
 $g_{\cdot}(t,\cdot,\cdot)$ is constant in $t$.
\end{proposition}

\noindent
\begin{proof}
a) Assume sub time-consistency holds.
By~\eqref{eq: dual-repr-bsvie-family}, the penalty term in the dual representation of $\rho_{tu}$ is given by
\begin{equation*}
\alpha^{\mathcal{G}}_{tu} (Q_t)= E_{Q_t} \left[\left. \int_t ^u g_u^*(t,v,q(t,v)) dv \right\vert \mathcal{F}_t \right]
\end{equation*}
for any $Q_t \in \mathcal{Q}_{tu}$.
Let $s, t, u \in [0,T]$ with $s \leq t \leq u$ and let $Q_s \in \mathcal{Q}_{st}, Q_t \in \mathcal{Q}_{tu}$ be fixed arbitrarily. Set now $\bar{Q}$ the pasting of $Q_s$ on $[s,t]$ and of $Q_t$ on $[t,u]$, hence $\bar{Q} \in \mathcal{Q}_{su}$. Denote by $q(s,v)$, $q(t,v)$ and $\bar{q}(s,v)$ the corresponding processes as in~\eqref{eq: definition-Qt}.
From the characterization of the penalty term for sub time-consistency in Proposition~\ref{prop: charact-sub time-cons}(iii) it follows that
\begin{eqnarray*}
\hspace{-4mm} &&{\displaystyle E_{\bar{Q}} \left[\left. \int_s ^u g_u^*(s,v,\bar{q}(s,v)) dv \right\vert \mathcal{F}_s \right] } \\
\hspace{-4mm} &&{\displaystyle \leq  E_{\bar{Q}} \left[\left. \int_s ^t g_t^*(s,v,\bar{q}(s,v)) dv \right\vert \mathcal{F}_s \right] + E_{\bar{Q}} \left[ \left. E_{\bar{Q}} \left[\left. \int_t ^u g_u^*(t,v,\bar{q}(t,v)) dv \right\vert \mathcal{F}_t \right] \right\vert \mathcal{F}_s \right]} ,
\end{eqnarray*}
hence
\begin{eqnarray}
\hspace{-2mm} {\displaystyle 0 \leq} && \hspace{-4mm} {\displaystyle E_{Q_s} \left[ \left. \int_s ^t \left[g_t^*(s,v,\bar{q}(s,v))- g_u^*(s,v,\bar{q}(s,v))\right] dv \right\vert \mathcal{F}_s \right]} \notag\\
\hspace{-2mm} && \hspace{-4mm} {\displaystyle +E_{Q_s} \left[\left. E_{Q_t} \left[ \left. \int_t ^u \left[g_u^*(t,v,\bar{q}(t,v))- g_u^*(s,v,\bar{q}(s,v))\right] dv \right\vert \mathcal{F}_t \right] \right\vert \mathcal{F}_s \right]}. \label{eq: condition g-family-bsvie}
\end{eqnarray}
Since~\eqref{eq: condition g-family-bsvie} should hold for any $s \leq t \leq u$ and any $Q_s \in \mathcal{Q}_{st}, Q_t \in \mathcal{Q}_{tu}$, it follows that
\begin{equation*}
\left\{
\begin{array}{rl}
g_t^*(s,v,\bar{q}) \geq  g_u^*(s,v,\bar{q}), &\quad \mbox{ for any } s \leq v \leq t \leq u \mbox { and } \bar{q} \in \Bbb R^d \\
g_u^*(t,v,\bar{q}) \geq  g_u^*(s,v,\bar{q}), &\quad \mbox{ for any } s \leq t \leq v \leq u \mbox { and } \bar{q} \in \Bbb R^d
\end{array}
\right.
\end{equation*}
Hence, $g_u$ is increasing in $u$ and $g_{\cdot}(t,\cdot,\cdot)$ is decreasing in $t$.
\smallskip

Conversely, sub time-consistency of $(\rho_{tu})_{t,u}$ induced by a family of BSVIEs can be written in the following notation
\begin{equation} 
\mathcal{E}^{g_t,V} (\mathcal{E}^{g_u,V} (-X \vert \mathcal{F}_t) \vert \mathcal{F}_s) \leq \mathcal{E}^{g_u,V} (-X \vert \mathcal{F}_s)
\end{equation}
for any $s \leq t \leq u$ and $X \in L^2(\mathcal{F}_u)$.
By~\eqref{eq: BSVIEg-family}, the right-hand and left-hand sides of the previous equation can be rewritten, respectively, as follows:
\begin{equation}  \label{eq: sub tc-proof-2-family}
\mathcal{E}^{g_u,V} (-X \vert \mathcal{F}_s)  = -X +\int_s ^u g_u(s,v,\hat{Z}(s,v)) \, dv - \int_s^u \hat{Z}(s,v) \, dB_v
\end{equation}
and
\begin{eqnarray}
\hspace{-5mm}&&\mathcal{E}^{g_t,V} (\mathcal{E}^{g_u,V} (-X \vert \mathcal{F}_t) \vert \mathcal{F}_s) \notag \\
\hspace{-5mm}&=& -X +\int_t ^u g_u(t,v,\tilde{Z}(t,v)) \, dv - \int_t^u \tilde{Z}(t,v) \, dB_v  \notag\\
\hspace{-5mm}&& +\int_s ^t g_t(s,v,Z(s,v)) \, dv -\int_s ^t  Z(s,v) \, dB_v \notag \\
\hspace{-5mm}&=& -X +\int_s ^u \left[ g_t(s,v,Z(s,v))1_{[s,t]}(v)+ g_u(t,v,\tilde{Z}(t,v))1_{(t,u]}(v) \right]\, dv \notag \\
\hspace{-5mm}&&- \int_s^u \left[ Z(t,v)1_{[s,t]}(v)+ \tilde{Z}(t,v) 1_{(t,u]}(v) \right] \, dB_v. \label{eq: sub tc-proof-bsvie-proof-1}
\end{eqnarray}
Furthermore,~\eqref{eq: sub tc-proof-bsvie-proof-1} becomes
\begin{eqnarray}
&\:&\mathcal{E}^{g_t,V} (\mathcal{E}^{g_u,V} (-X \vert \mathcal{F}_t) \vert \mathcal{F}_s) \notag \\
&\leq & -X +\int_s ^u \left[ g_t(s,v,Z(s,v))1_{[s,t]}(v)+ g_u(s,v,\tilde{Z}(t,v))1_{(t,u]}(v) \right]\, dv  \notag \\
&\:& - \int_s^u \left[ Z(t,v)1_{[s,t]}(v)+ \tilde{Z}(t,v) 1_{(t,u]}(v) \right] \, dB_v \notag\\
&\leq & -X +\int_s ^u \left[ g_u(s,v,Z(s,v))1_{[s,t]}(v)+ g_u(s,v,\tilde{Z}(t,v))1_{(t,u]}(v) \right]\, dv \notag \\
&& - \int_s^u \left[ Z(t,v)1_{[s,t]}(v)+ \tilde{Z}(t,v) 1_{(t,u]}(v) \right] \, dB_v, \label{eq: sub tc-proof-bsvie-family-3}
\end{eqnarray}
where the former inequality is due to decreasing monotonicity of $g_{\cdot}(t,\cdot,\cdot)$, the latter from increasing monotonicity of the family of drivers.
By setting
\begin{equation*}
\bar{Z}(s,v)=
Z(s,v) 1_{[s,t]}(v)+
\tilde{Z}(t,v) 1_{(t,u]}(v),
\end{equation*}
~\eqref{eq: sub tc-proof-bsvie-family-3} becomes
\begin{equation} \label{eq: sub tc-proof-bsvie-family-4}
\mathcal{E}^{g_t} (\mathcal{E}^{g_u} (-X \vert \mathcal{F}_t) \vert \mathcal{F}_s) \leq -X +\int_s ^u g_u(s,v,\bar{Z}(s,v)) \, dv - \int_s^u \bar{Z}(s,v) \, dB_v.
\end{equation}
Sub time-consistency then follows by comparing~\eqref{eq: sub tc-proof-bsvie-family-4} and~\eqref{eq: sub tc-proof-2-family} and
by the uniqueness of the solution of a BSVIE.\smallskip

\noindent
b)
The case of time-consistency can be obtained by replacing inequalities
with equalities in the proof above.
\end{proof}

\subsubsection{H-longevity and families of BSVIEs}

The following result provides sufficient conditions for h-longevity, similarly to Proposition~\ref{prop: longevity-BSDE-gt} for BSDEs.
\begin{proposition} \label{prop: longevity-BSVIE-family}
a) If $\mathcal{G}$ is an increasing family of drivers, $g_{\cdot}(t, \cdot, \cdot)$ is decreasing in $t$, and $\rho_{tu} (0) \geq 0$ for any $t \leq u$, then $(\rho_{tu})_{t,u}$ satisfies h-longevity.

\noindent
b) If $\mathcal{G}$ is an increasing family of drivers and $g_u \geq 0$ for any $u$, then h-longevity holds.
\end{proposition}

\noindent
\begin{proof}
a) follows by Remark~\ref{rem: tc and longevity} and Proposition~\ref{prop: sub-tc-bsvie-family}.
\smallskip
\noindent
b) Let $s \leq t \leq u$ and $X \in L^2(\mathcal{F}_t)$ be fixed arbitrarily.
The risk measures $\rho_{st}$ and $\rho_{su}$ satisfy, respectively, the following BSVIEs:
\begin{eqnarray*}
\rho_{st}(X)&=& -X+ \int_s^t g_t(s,v,Z^t(s,v)) dv - \int_s^t Z ^t (s,v) dB_v \\
\rho_{su}(X)&=& -X+ \int_s^u g_u(s,v,Z^u (s,v)) dv - \int_s^u Z ^u (s,v) dB_v.
\end{eqnarray*}
Set now
\begin{equation*}
\bar{Z}^{t,u} (s,v)=\left\{
\begin{array}{rl}
Z^t (s,v);& s \leq v \leq t \\
0;& t<v \leq u
\end{array}
\right. ;
\quad \widetilde{Z} (s,v)= Z^u (s,v) - \bar{Z}^{t,u} (s,v).
\end{equation*}
Then
\begin{eqnarray}
\delta \rho_s &=& \rho_{su}(X)-\rho_{st}(X) \notag\\
&=&  \int_s^u [g_u(s,v,Z^u (s,v))- g_t(s,v,\bar{Z}^{t,u}(s,v))] dv + \int_t^u g_t(s,v,0) dv \notag \\
&& - \int_s^u [Z ^u (s,v)- \bar{Z}^{t,u} (s,v)] dB_v  \notag \\
&\geq &  \int_s^u [g_u(s,v,Z^u (s,v))- g_u(s,v,\bar{Z}^{t,u}(s,v))] dv  - \int_s^u \widetilde{Z} (s,v) dB_v  \notag \\
&=&  \int_s^u \Delta_z g_u(s,v) \cdot \widetilde{Z} (s,v) dv - \int_s^u \widetilde{Z} (s,v) dB_v, \label{eq: bsvie-family-longevity-1}
\end{eqnarray}
where the inequality is due to the hypothesis on the drivers and where
\begin{equation*}
\Delta_z^i g_u(s,v) \triangleq \frac{g_u(s,v,Z^u(s,v))- g_u(s,v,\bar{Z}^{t,u}(s,v))}{d \left( Z^{u,i} (s,v)- \bar{Z}^{t,u, i} (s,v)\right)} 1_{\{Z^{u,i} (s,v)\neq \bar{Z}^{t,u,i} (s,v)\}}
\end{equation*}
for $i=1,...,d$. By Girsanov Theorem,~\eqref{eq: bsvie-family-longevity-1} becomes
\begin{eqnarray*}
\delta \rho_s &\geq & - \int_s^u \widetilde{Z}(s,v) dB^{\widetilde{Q}_s}_v,
\end{eqnarray*}
where $\frac{d \widetilde{Q}_s}{dP}\triangleq \exp \left\{ - \frac 12 \int_s^u \vert\Delta_z g_u(s,v)\vert^2 dv + \int_s^u \Delta_z g_u(s,v) dB_v \right\}$ and $B^{\widetilde{Q}_s}_v \triangleq B_v-B_s - \int_s^v \Delta_z g_u(s,v) dv$, for $v \geq s$, is a Brownian motion with respect to $\widetilde{Q}_s$ with initial value $B^{\widetilde{Q}_s}_s=0$. Hence, by taking the conditional expectation with respect to $\widetilde{Q}_s$,
\begin{equation*}
\delta \rho_s \geq  E_{\widetilde{Q}_s} \left[\left. - \int_s^u \widetilde{Z}(s,v) dB^{\widetilde{Q}_s}_v \right| \mathcal{F}_s\right]=0.
\end{equation*}
This completes the proof.
\end{proof}

In the spirit of Example~\ref{ex: BSDE g non normalized - entropic}, we can easily extend Example~\ref{ex: linear Volterra} and Example~\ref{ex: quadratic Volterra} to cover the case of families of BSVIEs.

\vspace{5mm}
\noindent
{\bf Acknowledgements.}
We thank Tomasz Bielecki, Matteo Burzoni, Igor Cialenco, Alessandro Doldi, Nicole El Karoui, Mario Ghossoub, Michael Kupper, Felix Liebrich, Max Nendel, and Frank Riedel for their interest and comments.
The research leading to these results has received funding from the Research Council of Norway (RCN) within the project {\it STORM - Stochastics for time-space risk models}  (nr. 274410). In particular, the second author thanks this research group for the warm hospitality during her visits. The second author is member of Gruppo Nazionale per l'Analisi Matematica, la Probabilit\`{a} e le loro Applicazioni (GNAMPA), Italy.

\end{document}